\documentclass[a4paper,12pt]{amsart}
\usepackage{amsthm,amsmath,amssymb}
\usepackage{graphicx}
\usepackage{paralist}
\usepackage[notext]{stix2}
\usepackage[arrow]{xy}
\usepackage{pb-diagram,pb-xy}
\usepackage{curves}
\usepackage{xcolor}
\allowdisplaybreaks[4]
\theoremstyle{plain}
  \newtheorem{thm}{Theorem}[section]
  \newtheorem{lem}[thm]{Lemma}
  \newtheorem{prop}[thm]{Proposition}
  \newtheorem{cor}[thm]{Corollary}
\theoremstyle{definition}
  \newtheorem{defn}[thm]{Defninition}
  \newtheorem{expl}[thm]{Example}
\theoremstyle{remark}
  \newtheorem{rem}[thm]{Remark}
\numberwithin{equation}{section}

\def\numeric{\mathbb{N}}

\def\real{\mathbb{R}}

\def\I{\mathbb{I}}
\def\art{\text{art.}}
\def\homeo{\approx}
\def\fracinline#1/#2{\mbox{\raise0.5ex\hbox{\footnotesize$#1$}{\footnotesize\hskip-.1em$/$\hskip-.1em}\raise-0.5ex\hbox{\footnotesize$#2$}}}
\def\timesab#1#2{ \,{}_{#1\!\!} \times_{\,#2} }
\def\colim{\operatorname{colim}}
\def\comp{\smash{\lower-.1ex\hbox{\scriptsize$\circ$\,}}}
\def\proj{\mathrm{pr}}

\def\id{\mathrm{1}}
\def\join{\,\ast\,}

\def\hooklongrightarrow{\lhook\joinrel\longrightarrow}

\def\der#1by#2{\frac{\operatorname{\mathit{d}}\hspace{-0.1mm}#1}{\operatorname{\mathit{d}}\hspace{-0.1mm}#2}}
\def\nder#1by#2times#3{\frac{\operatorname{\mathit{d}}^{#3}\hspace{-0.2mm}#1}{\operatorname{\mathit{d}}\hspace{-0.4mm}{#2\,}^{#3}}}
\def\pder#1by#2{\frac{\operatorname{\partial}\hspace{-0.1mm}#1}{\operatorname{\partial}\hspace{-0.1mm}#2}}
\def\npder#1by#2times#3{\frac{\operatorname{{\partial\,}^{#3}}\hspace{-0.2mm}#1}{\operatorname{\partial}\hspace{-0.1mm}{#2}^{#3}}}
\def\diff#1{\ifx#1(\operatorname{\mathit{d}}\hspace{0.25mm}#1\else\operatorname{\mathit{d}}\hspace{-0.5mm}#1\fi}
\def\pdiff#1{\ifx#1(\operatorname{\partial}\hspace{0.25mm}#1\else\operatorname{\partial}\hspace{-0.25mm}#1\fi}
\def\emptyarg{}

\def\Img#1{\operatorname{Im}\hskip1.5pt(\hskip1pt#1)}
\def\midvert{ \,\,\mathstrut\vrule\,\, }
\def\ad#1{\def\thisarg{#1}\mathrm{ad}\ifx\thisarg\emptyarg\else(#1)\fi}
\def\const#1{\def\thisarg{#1}\operatorname{\iota}\ifx\thisarg\emptyarg\else(\hskip1pt#1)\fi}
\def\Path#1{\def\thisarg{#1}\ifx\thisarg\emptyarg{\operatorname{\mathcal{P}}}\else{\operatorname{\mathcal{P}}(#1)}\fi}
\def\Category#1{{\sf{#1}}}
\def\Site#1{{\sf{#1}}}
\def\Object#1{\operatorname{Obj}\,(\text{\small\(#1\)})}
\def\Morphism#1{\operatorname{Mor}_{\text{\small\,\(#1\)\,}}}
\def\Covering#1{\operatorname{Cov}_{\text{\small\,\(#1\)\,}}}
\def\manifold{\Category{Manifold}}
\def\diffeology{\Category{Diffeology}}
\def\topology{\Category{Topology}}

\def\chen{\Category{Chen}}
\def\sets{\Site{Set}}
\def\euclidean#1{\def\thisarg{#1}\ifx\thisarg\emptyarg{\Site{Euclidean}}\else{\Site{C^{#1}\text{-}Euclidean}}\fi}
\def\domain#1{\def\thisarg{#1}\ifx\thisarg\emptyarg{\Site{Domain}}\else{\Site{C^{#1}\text{-}Domain}}\fi}
\def\convex#1{\def\thisarg{#1}\Site{\ifx\thisarg\emptyarg\else{C^{r}\text{-}}\fi{Convex}}}
\def\polyhedron#1{\def\thisarg{#1}\Site{\ifx\thisarg\emptyarg\else{C^{r}\text{-}}\fi{Polyhedron}}}

\def\homeo{\approx}

\def\Loop#1{\operatorname{\mathcal{L}}(#1)}
\def\Paths{\operatorname{Paths}}
\def\stPaths{\operatorname{stPaths}}
\def\Map{\operatorname{Map}}
\def\Img{\operatorname{Im}}

\def\ast{\hbox{\footnotesize$*$}}
\def\differentiabletxt{smooth}
\def\smoothtxt{smooth}
\def\Smoothtxt{Smooth}
\def\Cubic{C}
\def\cubictxt{cubic}
\def\Cubictxt{Cubic}
\def\cubicsettxt{cubic set}

\def\cubicsetstxt{cubic sets}
\newenvironment{enumerate*}{\vskip.5ex
\begin{inparaenum}[(1)\hspace{.2em}]
}{\end{inparaenum}\vskip.5ex\noindent}
\def\vitem{\par\hskip1em\item}
\def\hitem{\hskip1em\item}
\begin{document}
\ifdefined\expand
\baselineskip21pt
\else
\ifdefined\narrow
\baselineskip15pt
\else
\baselineskip18pt
\fi
\fi
%
%
\title{Smooth $A_{\infty}$-form on a diffeological loop space}
%
%
\author{Norio IWASE}
\email{iwase@math.kyushu-u.ac.jp\vskip-1ex}
%
%
\address
{Faculty of Mathematics,
 Kyushu University,
 Fukuoka 819-0395, Japan\vskip-1ex}
%
%
%
%
%
%
\begin{abstract}
To construct an $A_{\infty}$-form for a loop space in the category of diffeological spaces, we have two minor problems. Firstly, the concatenation of paths in the category of diffeological spaces needs a small technical trick (see P.~I-Zemmour \cite{MR3025051}), which apparently restricts the number of iterations of concatenations. Secondly, we do not know a natural smooth decomposition of an associahedron as a simplicial or a cubical complex.
To resolve these difficulties, we introduce a notion of a $q$-cubic set which enjoys good properties on dimensions and representabilities, and show, using it, that the smooth loop space of a reflexive diffeological space is a h-unital smooth $A_{\infty}$-space.
In appendix, we show an alternative solution by modifying the concatenation to be stable without assuming reflexivity for spaces nor stability for paths.
\end{abstract}
%
%
\maketitle

\section*{Introduction}

A site is a concrete category with a `coverage' assigning a `covering family' to each object.
For a site $\Site{C}$, we denote by $\Object{\Site{C}}$ the class of objects, by $\Morphism{\Site{C}}(A,B)$ the set of morphisms from $A$ to $B$, and by $\Covering{\Site{C}}(U)$ the set of covering families on $U \in \Object{\Site{C}}$.
We denote by $\sets$ the category of sets and maps between sets.
For a given set $X$, we have two contravariant functors $\mathcal{M}_{X}, \,\mathcal{K}_{X} : \Site{C} \to \sets$ defined by
\begin{enumerate}
\item
$\mathcal{M}_{X}(U) = \Map(U,X)$ the set of maps from $U$ to $X$ and
\vspace{.5ex}\item
$\mathcal{K}_{X}(U) = \{\,P \in \mathcal{M}_{X}(U) \midvert \text{$P : U \to X$ is locally constant}\,\}$,
\end{enumerate}
where we say $P : U \to X$ is locally constant, if there exists a covering family $\{\,g_{\alpha} : V_{\alpha} \to U\,\}_{\alpha \in \Lambda}$ of $U$ such that $P{\comp}g_{\alpha}$ is constant for any $\alpha \in \Lambda$.

In \cite{MR380859,MR377960,MR454968,MR842915}, K.~T.~Chen introduced a site $\convex{}$ which is a category of convex sets with non-void interiors in $\real^{n}$ for some $n \!\ge\! 0$, and {\smoothtxt} functions between them in the ordinary sense (see \cite{MR1471480}), with a `coverage' assigning a `covering family' to each convex set with non-void interior, which is the set of open coverings by interiors of convex sets.

In \cite{MR607688}, J.~M.~Souriau introduced a similar but a slightly more sophisticated site $\domain{}$ which is a category of open sets in $\real^{n}$ for some $n \!\ge\! 0$, and {\smoothtxt} functions between them in the ordinary sense, with a `coverage' assigning a `covering family' to each open set, which is the set of open coverings in the usual sense.

We call a pair $(X,\mathcal{D}_{X})$ a diffeological space, if it satisfies the following conditions.
\begin{enumerate}
\item[(D1)]\label{defn:smoothology'-1}
$X$ is a set and $\mathcal{D}_{X} : \domain{} \to \sets$ is a contravariant functor.
\vspace{.5ex}\item[(D2)]\label{defn:smoothology'-2}
For any $U \in \Object{\domain{}}$, $\mathcal{K}_{X}(U)$ $\subset$ $\mathcal{D}_{X}(U)$ $\subset$ $\mathcal{M}_{X}(U)$.
\vspace{.5ex}\item[(D3)]\label{defn:smoothology'-3}
For any $U \!\in\! \Object{\domain{}}$ and any $P \in \mathcal{M}_{X}(U)$, 
$P \in \mathcal{D}_{X}(U)$ if there exists $\{U_{\alpha}\}_{\alpha\in\Lambda} \in \Covering{\domain{}}(U)$ such that $P\vert_{U_{\alpha}} \in \mathcal{D}_{X}(U_{\alpha})$ for all $\alpha \in \Lambda$.
\end{enumerate}
A map $f : X \to Y$ is said to be {\differentiabletxt}, if 
the natural transformation $f_{\!\ast} : \mathcal{M}_{X} \to \mathcal{M}_{Y}$ satisfies $f_{\!\ast}(\mathcal{D}_{X}(U)) \subset \mathcal{D}_{Y}(U)$ for any $U \!\in\! \Object{\domain{}}$.
We denote by $\diffeology$, the category of diffeological spaces and {\differentiabletxt} maps between diffeological spaces.
An element of $\mathcal{D}_{X}(U)$ is called a plot of $X$ on $U$, and \vspace{.5ex}$\mathcal{D} = \bigcup_{U \in \Object{\domain{}}}\mathcal{D}_{X}(U)$ is called a `diffeology' on $X$.
If we replace the site $\domain{}$ by the site $\convex{}$, we obtain Chen's smooth category denoted by $\chen$.
From now on, we discuss in the smooth category $\diffeology{}$, rather than $\chen{}$, while we believe that entirely similar arguments can be performed also in $\chen{}$.
Let $\numeric$ be the set of non-negative integers.

\begin{rem}
For any set $X$, both $\mathcal{K}_{X}$ and $\mathcal{M}_{X}$ give diffeologies on $X$.
In fact, $\mathcal{K}_{X}$ gives the finest diffeology on $X$ and $\mathcal{M}_{X}$ gives the coarsest diffeology on $X$ (see \cite[1.18]{MR3025051}).
\end{rem}

\begin{rem}
Let $\euclidean{}$ be the full-subcategory of $\domain{}$ consisting of all Euclidean spaces of some dimension $\in \numeric$.
Even if we replace $\domain{}$ with $\euclidean{}$ in the definition of $\diffeology$, we recover $\diffeology$ itself (see \cite[Exercise 3]{MR3025051}).
\end{rem}
\begin{rem}
For $n \!\in\! \numeric$, let $\euclidean{}_{n}$ be the full-subcategory of $\euclidean{}$ consisting of all Euclidean spaces of dimension up to $n$.
If we replace $\domain{}$ with $\euclidean{}_{0}$ in the definition of $\diffeology$, we must obtain $\sets$ the category of sets.
If we replace $\domain{}$ with $\euclidean{}_{1}$ in the definition of $\diffeology$, we must obtain the category of diffeological spaces with `wire-diffeology' (see \cite[\art 1.10]{MR3025051}).
\end{rem}

In this paper, a manifold is assumed to be paracompact.
We denote by $\manifold$ the category of {\smoothtxt} manifolds and {\smoothtxt} maps between them which can be embedded into $\diffeology{}$ as a full subcategory (see \cite{MR3025051}).
One of the advantage to expand our playground to $\diffeology{}$ than to restrict ourselves in $\manifold$ is that the category $\diffeology{}$ is cartesian-closed, complete and cocomplete (see \cite{MR3025051}).

The path space in $\diffeology{}$ is defined using the real line $\real$ in place of the closed interval $[0,1]$ (see \cite[Chapter 5]{MR3025051}).
This definition gives a nice diffeology on a path space, while it causes a technical issue on concatenation:
$$
\Paths(X)=C^{\infty}(\real,X) \ \text{($= \mathcal{D}_{X}(\real)$ as a set)}
$$
A work-around can easily be found as in \cite[\art 5.4]{MR3025051} by compressing the moving part into an open subinterval $(\varepsilon,1{-}\varepsilon) \subset (0,1) \subset \real$, where $0 < \varepsilon \ll 1$:
\begin{align*}&
\stPaths_{\varepsilon}(X)=\{\, u \in \Paths(X) \midvert \forall\,{t \!\le\! \varepsilon}\, \,u(t)=u(0) \ \& \ \forall\,{t \!\ge\! 1{-}\varepsilon}\, \,u(t)=u(1) \,\}
\end{align*}

On the other hand, if we consider $A_{\infty}$-form of concatenations using $\stPaths_{\varepsilon}(X)$, we need some more tricks to concatenate many paths.
In this paper, we adopt slightly different ways to consider a smooth (h-unital) $A_{\infty}$-form for a concatenation.

Let $\Path{X} = \{\,u \in \mid u{\comp}\pi_{set}=u\,\}$, where $\pi_{set} : \real \to \real$ is a continuous idempotent, i.e, $\pi_{set}{\comp}\pi_{set}=\pi_{set}$, which is defined as follows:
$$
\pi_{set}(t)=\max\{0,\min\{1,t\}\}=\min\{1,\max\{0,t\}\}.
$$
Then by definition, $\pi_{set}$ enjoys the following properties.
\begin{enumerate*}
\item\label{property:pi1} $\pi_{set}(t)=0$, $t \le 0$,
\hitem\label{property:pi2} $\pi_{set}(t)+\pi_{set}(1{-}t)=1$,
\hitem\label{property:pi3} $\pi_{set}(t)=t$, $0 < t < 1$.
\end{enumerate*}\par\noindent
Then by (\ref{property:pi1}) and (\ref{property:pi2}) above, we have $\pi_{set}(t)=1$, $t \ge 1$ as well.

\section{Basic properties on subductions}

Let us recall basic properties on subductions in \diffeology{} used in this paper.

\begin{lem}\label{lem:subduction-product}
Let $\varpi_{1} : K \to X$ and $\varpi_{2} : L \to Y$ be two subductions.
Then, $\varpi=\varpi_{1} \times \varpi_{2} : K \times L \to X \times Y$ is also a subduction.
\end{lem}
\begin{proof}
This can be obtained using \cite[Lemma 2.5]{MR3913971}.
But we give here a direct proof: clearly, $\varpi$ is a smooth surjection, and so we are left to show that a plot on $X \times Y$ can be pulled back to $K \times L$ locally.
Let $P : V \to X \times Y$ be a plot.
We denote by $\proj_{k}$ the canonical projection from a product to its $k$-th factor, $k=1,2$.
Then, $P_{\!1} = \proj_{1}{\comp}P : V \to X$ and $P_{\!2} = \proj_{2}{\comp}P : V \to Y$ are plots, for each $\proj_{k}$ is smooth.
Since $\varpi_{1} : K \to X$ and $\varpi_{2} : L \to Y$ are subductions, there is an open covering $\{V_{\!\alpha}\}$ of $V$ such that there are plots $Q_{1\alpha} : V_{\!\alpha} \to K$ and $Q_{2\alpha} : V_{\!\alpha} \to L$ satisfying $P_{\!k}|_{V_{\!\alpha}}=\varpi_{k}{\comp}Q_{k\alpha}$, $k=1,2$.
Using the data $Q_{1\alpha}$ and $Q_{2\alpha}$, we obtain a smooth map $Q_{\alpha} : V_{\!\alpha} \to K \times L$ satisfying $Q_{k\alpha} = \proj_{k}{\comp}Q_{\alpha}$, $k=1,2$.
Thus we obtain $\varpi{\comp}Q_{\alpha}=P|_{V_{\!\alpha}}$, and hence $X \times Y$ has the push-forward diffeology by $\varpi : K \times L \to X \times Y$.
\end{proof}

From now on, we assume that $L$, $Y$ and $X$ are diffeological spaces, and that $\varpi : L \to Y$ is a subduction.
Here, we remark that $\dim{Y} \le \dim{L}$ diffeologically.

\begin{lem}\label{lem:subduction-universality}
For a map $g : Y \to X$, $g$ is smooth iff $g{\comp}\varpi : L \to X$ is smooth.
\end{lem}
\begin{proof}
It is sufficient to show that $g$ is smooth if $g{\comp}\varpi : L \to X$ is smooth: let $P : U \to Y$ be a plot.
Since $\varpi$ is a subduction, there is an open covering $\{V_{\alpha}\}$ of $U$ and plots $\{P_{\alpha} : V_{\alpha} \to L\}$ such that $P|_{V_{\alpha}}$ $=$ $\varpi{\comp}P_{\alpha}$ for all $\alpha$.
Since $g{\comp}P|_{V_{\alpha}}$ $=$ $(g{\comp}\varpi){\comp}P_{\alpha} : V_{\alpha} \to X$ is smooth for all $\alpha$, so is $g{\comp}P : U \to X$.
Thus $g$ is smooth.
\end{proof}

\begin{prop}\label{prop:subduction-dual}
$\varpi$ induces an induction $\varpi^{*} : C^{\infty}(Y,X) \to C^{\infty}(L,X)$.
\end{prop}
\begin{proof}
Since $\varpi : L \to Y$ is a smooth surjection, $\varpi^{*} : C^{\infty}(Y,X) \to C^{\infty}(L,X)$ is a smooth injection.
Now we are left to show that a plot in $\Img\varpi^{*} \subset C^{\infty}(L,X)$ can be pulled back to $C^{\infty}(Y,X)$:
let $P : U \to \Img{\varpi^{*}} \subset C^{\infty}(L,X)$ be a plot.
Then, for any $\mathbb{u} \in U$, there is $F_{\!\mathbb{u}} \in C^{\infty}(Y,X)$ such that $P(\mathbb{u}) = F_{\!\mathbb{u}}{\comp}\varpi$.
For $t, \,s \in L$ with $\varpi(t)=\varpi(s)$ and $\mathbb{u} \in U$, the adjoint $\widehat{P} : U \times L \to X$ of $P$ is a smooth map satisfying $\widehat{P}(\mathbb{u},t) = P(\mathbb{u})(t)  = F_{\!\mathbb{u}}{\comp}\varpi(t) = F_{\!\mathbb{u}}{\comp}\varpi(s) = P(\mathbb{u})(s) = \widehat{P}(\mathbb{u},s)$.
Thus the smooth map $\widehat{P} : U \times L \to X$ induces a map $\widehat{Q} : U \times Y \to X$ such that $\widehat{Q}{\comp}(\id \times \varpi) = \widehat{P}$, where $\id \times \varpi : U \times L \to U \times Y$ is a subduction by Lemma \ref{lem:subduction-product}.
By Lemma \ref{lem:subduction-universality}, $\widehat{Q}$ is smooth and hence its adjoint $Q : U \to C^{\infty}(Y,X)$ is a plot satisfying $(\varpi^{*}{\comp}Q(\mathbb{u}))(t)$ $=$ $Q(\mathbb{u})(\varpi(t))$ $=$ $\widehat{Q}(\mathbb{u},\varpi(t))$ $=$ $\widehat{Q}{\comp}(\id \times \varpi)(\mathbb{u},t)$ $=$ $\widehat{P}(\mathbb{u},t)$ $=$ $P(\mathbb{u})(t)$, $(\mathbb{u},t) \in U \times L$, which implies $\varpi^{*}{\comp}Q = P$, and we have done.
\end{proof}

We further assume that $Y$ is a diffeological quotient $L/\varpi_{set}$ by a relation $\varpi_{set}$ on $L$, i.e, $Y=\{\,y \mid \exists\,{t \in L} \ y=\widehat\varpi_{set}(t)\,\}$ and $\varpi(t)=\widehat\varpi_{set}(t)$, where $\widehat\varpi_{set}$ is an equivalence relation on $L$ generated by $\varpi_{set}$.
Hence $\varpi{\comp}\varpi_{set}=\varpi$ as relations from $L$ to $Y$.

\begin{prop}\label{prop:diffeologicalquotient}
$\Img{\varpi^{*}}=\{\,f \in C^{\infty}(L,X) \mid f{\comp}\varpi_{set}=f \ \text{as relations}\,\}$.
\end{prop}
\begin{proof}
We show that $\{\,f \in C^{\infty}(L,X) \mid f{\comp}\varpi_{set}=f \,\} \subset \Img{\varpi^{*}}$.
If $f{\comp}\varpi_{set}=f \in C^{\infty}(L,X)$, then $f$ induces a map $g : Y \to X$ such that $f=g{\comp}\varpi$.
By Lemma \ref{lem:subduction-universality}, $g$ is smooth, and hence $f \in \Img{\varpi^{*}}$.
The converse is clear by $\varpi{\comp}\varpi_{set}=\varpi$.
\end{proof}
If the relation $\varpi_{set}$ is a continuous idempotent on $L$, then $Y$ is topologically the same as $\Img\varpi_{set} \subset L$, while $Y$ might not be a diffeological subspace of $L$.

\section{{\Cubictxt} complex in \topology{}}

In \topology{}, we use the symbol $\I_{top}$ for the topological subspace $[0,1]$ of $\real$.
We remark that the topology of $\I_{top}=[0,1]$ is the same as the quotient topology induced by a continuous map $\pi_{top} : \real \to \I_{top}$ given by $\pi_{top}(t)=\pi_{set}(t) \in [0,1]$.

Now we introduce a generalised notion of a simplicial or cubical complex using an idea of a {\cubicsettxt}:
a $q$-cubic set $\sigma$ in $\real^{n}$ is defined as a convex body in some affine subspace $L_{\sigma}$ in $\real^{n}$, inductively on $q$, $-1 \!\le\! q \!\le\! n$ (see also \cite{MR3991180}).
\begin{enumerate}
\item%
The $-1$-{\cubicsettxt} in $\real^{n}$ is the empty set $\emptyset \subset \real^{n}$. In this case, $L_{\emptyset}=\emptyset$.
\vspace{.5ex}\item
A $0$-cubic set in $\real^{n}$ is a point $\mathbb{p} \in \real^{n}$.
In this case, $L_{\mathbb{p}}=\{\mathbb{p}\}$.
\vspace{.5ex}\item%
Let $\sigma_{1}$ and $\sigma_{2}$ be respectively $q_{1}$-cubic and $q_{2}$-cubic sets in $\real^{n}$ with $-1 \le q{-}1 \le q_{1}{+}q_{2} \le q \le n$, where $\sigma_{1}$ and $\sigma_{2}$ are convex bodies in affine subspaces $L_{1}$ and $L_{2}$, respectively.
Let $V_{1}$ and $V_{2}$ be vector subspaces of $\real^{n}$ such that $V_{1} \cap V_{2} = \{\mathbb{0}\}$, 
$L_{1} = \mathbb{a}_{1} \!+\! V_{1}$ and $L_{2} = \mathbb{a}_{2} \!+\! V_{2}$ for some $\mathbb{a}_{1} \in \sigma_{1}$ and $\mathbb{a}_{2} \in \sigma_{2}$.
\begin{enumerate}
\vspace{.5ex}\item[($q_{1}\!+\!q_{2}=q{-}1$ and $L_{1} \!\cap\! L_{2} = \emptyset$ (or $\mathbb{a}_{2}\!-\!\mathbb{a}_{1} \not\in V_{1}\!+\!V_{2}$))]\label{cubic-const-join}\ \ 
The subset $\sigma_{1} \join \sigma_{2}$ $=$ $\{\,(1{-}t){\cdot}\mathbb{x}$ ${+}$ $t{\cdot}\mathbb{y} \,;\, \mathbb{x} \!\in\! \sigma_{1}, \,\mathbb{y} \!\in\! \sigma_{2}, \,t \!\in\! I_{top}\,\} \subset \real^{n}$ is a $q$-{\cubicsettxt} in $\real^{n}$.
In this case, we have a relative homeomorphism $\phi_{\sigma_{1},\sigma_{2}} : (\sigma_{1} \times I_{top} \times \sigma_{2} , \sigma_{1}  \times \{0,1\} \times \sigma_{2}) \to (\sigma_{1} \join \sigma_{2},\sigma_{1} \amalg \sigma_{2})$ given by $\phi_{\sigma_{1},\sigma_{2}}(\mathbb{x},t,\mathbb{y})=(1{-}t){\cdot}\mathbb{x}\!+\!t{\cdot}\mathbb{y}$.
\vspace{.5ex}\item[($q_{1}\!+\!q_{2}\!=\!q$ and $L_{1} \!\cap L_{2} \not= \emptyset$ (or $\mathbb{a}_{2}\!-\!\mathbb{a}_{1} \in V_{1}\!+\!V_{2}$))]\label{cubic-const-product}\ \ 
Let $L_{1} \!\cap L_{2} \!=\! \{\mathbb{a}\}$, $\mathbb{a} \!\in\! \real^{n}$.
Then the subset $\sigma_{1}{\timesab{L_{1}}{L_{2}}} \sigma_{2} = \{\mathbb{x} \,{+}\, \mathbb{y} \,{-}\, \mathbb{a} \,;\, \mathbb{x} \!\in\! \sigma_{1},\, \mathbb{y} \!\in\! \sigma_{2}\}$ is a $q$-{\cubicsettxt} in $\real^{n}$.
In this case, we have a homeomorphism $\psi_{\sigma_{1},\sigma_{2}} : \sigma_{1}  \times  \sigma_{2} \to \sigma_1 \timesab{L_{1}}{L_{2}} \sigma_{2}$ given by $\psi_{\sigma_{1},\sigma_{2}}(\mathbb{x},\mathbb{y})=\mathbb{x} \!+\! \mathbb{y} \!-\! \mathbb{a}$.
\end{enumerate}
\end{enumerate}

For each $n \geq 0$ and $q$ with $-1 \le q \le n$, we denote by $\Cubic(n)^{q}$ the set of all $q$-{\cubicsetstxt} in $\real^{n}$ and $\Cubic(n) = \{\emptyset\}\cup\underset{q \geq 0}{\cup}\Cubic(n)^{q}$.
Then the above construction yields two natural products: the join $\ast : C(n)^{q} \times C(n')^{q'} \to C(n{+}n'{+}1)^{q+q'+1}$ induced by (\ref{cubic-const-join}) above using $\real^{n} \homeo \real^{n} \times \{0\} \times \{\mathbb{0}\} = V_{1} \subset \real^{n} \times \real \times \real^{n'} \supset V_{2} = \{\mathbb{0}\} \times \{0\} \times \real^{n'} \homeo \real^{n'}$ with $\mathbb{a}_{t}=(\mathbb{0},t,\mathbb{0})$ for $t=1, 2$, and the product $\times : C(n)^{q} \times C(n')^{q'} \to C(n{+}n')^{q+q'}$ induced by (\ref{cubic-const-product}) above using $\real^{n} \homeo \real^{n} \times \{\mathbb{0}\} = V_{1} \subset \real^{n} \times \real^{n'} \supset V_{2} = \{\mathbb{0}\} \times \real^{n'} \homeo \real^{n'}$ with $\mathbb{a}=\mathbb{a}_{1}=\mathbb{a}_{2}=(\mathbb{0},\mathbb{0})$.

\smallskip

The notion of a face of a {\cubicsettxt} is inductively given as follows.
\begin{enumerate}
\item
Let $\sigma$ be a {\cubicsettxt}.
Then the emptyset $\emptyset$ and $\sigma$ itself are faces of $\sigma$.
\item
Let $\sigma_{1}$ and $\sigma_{2}$ be two {\cubicsettxt}.
Then we have the following.
\begin{enumerate}
\item
A face of $\sigma_{1} \join \sigma_{2}$ is expressed as $\tau_{1} \join \tau_{2}$ for some faces $\tau_{1}$ and $\tau_{2}$ of $\sigma_{1}$ and $\sigma_{2}$, respectively. Therefore $\sigma_{1}=\sigma_{1} \join \emptyset$ and $\sigma_{2} = \emptyset \join \sigma_{2}$ are faces of $\sigma_{1} \join \sigma_{2}$.
\item
A face of $\sigma_{1} \timesab{L_{1}}{L_{2}} \sigma_{2}$ is expressed as $\tau_{1} \timesab{L_{1}}{L_{2}} \tau_{2}$ for some faces $\tau_{1}$ and $\tau_{2}$ of $\sigma_{1}$ and $\sigma_{2}$, respectively.
\end{enumerate}
\end{enumerate}

We denote $\tau \prec \sigma$ if $\tau \in \Cubic(n)$ is a face of $\sigma \in \Cubic(n)$. 

An ordered subset $\mathbb{K} \subset \Cubic(n)$ is called a {\cubictxt} complex, if the following holds.
\begin{enumerate*}\setcounter{enumi}{-1}\vspace{.5ex}
\vitem
$\forall\,\tau, \,\sigma \in \mathbb{K} \ \ \tau\cap\sigma \in \Cubic(n), \ \tau\cap\sigma \prec \tau \ \text{and} \ \tau\cap\sigma \prec \sigma$.\vspace{.5ex}
\vitem
$\emptyset \in \mathbb{K}$,
\hitem
$\forall\,\tau \in \Cubic(n) \ \,\forall\,\sigma \in \mathbb{K} \ \ \tau\prec\sigma \implies \tau \in \mathbb{K}$,
\end{enumerate*}\vspace{.5ex}\noindent
A subset $\mathbb{L} \subset \mathbb{K}$ with the following properties is called a {\cubictxt} subcomplex of $\mathbb{K}$.\vspace{.5ex}
\begin{enumerate*}
\vitem
$\emptyset \in \mathbb{L}$,
\hitem
$\forall\,\tau \in \mathbb{K} \ \,\forall\,\sigma \in \mathbb{L} \ \ \tau\prec\sigma \implies \tau \in \mathbb{L}$.
\end{enumerate*}
Then we denote $\dim{\mathbb{K}}=\max\{\,\dim\sigma\mid\sigma\in\mathbb{K}\,\}$, where $\dim\sigma=q$ if $\sigma \in C(n)^{q}$.

For any $q$-{\cubictxt} set $\sigma \in K$, $\mathbb{K}(\sigma)=\{\,\tau \!\in\! C(n) \midvert \tau \preceqq \sigma\,\}$ for $q \ge -1$ and $\mathbb{K}(\dot\sigma) = \{\tau \!\in\! C(n) \midvert \tau \precneqq \sigma\}$ for $q \ge 0$ are {\cubictxt} subcomplexes of $\mathbb{K}$.
\begin{prop}
For any $q$-{\cubictxt} set $\sigma \in C(n)$, $q \ge 0$, we have $\partial\sigma = \vert{\mathbb{K}(\dot\sigma)}\vert$.
\end{prop}

For any two {\cubictxt} complexes $\mathbb{K} \subset C(n)$ and $\mathbb{L} \subset C(m)$, we obtain 
\begin{enumerate}
\item
$\mathbb{K} \join \mathbb{L} := \{\,\sigma \join \tau \midvert \sigma\!\in\!\mathbb{K}, \tau\!\in\!\mathbb{L}\,\} \subset C(n{+}m{+}1)$
\vspace{.5ex}\item
$\mathbb{K} \times \mathbb{L} := \{\,\sigma \times \tau \midvert \sigma\!\in\!\mathbb{K}, \tau\!\in\!\mathbb{L}\,\} \subset C(n{+}m)$
\end{enumerate}

\begin{prop}
For any two {\cubictxt} sets $\sigma, \tau \in C(n)$ and $0$-{\cubictxt} sets $a$, $b$, we have $(\sigma \join a) \times (\tau \join b)$ $\homeo$ $\mathbb{L} \join c$, where $\mathbb{L} = (\sigma \join a) \times \tau \cup \sigma \times (\tau \join b)$ and $c=(a,b)$.
\end{prop}

For any {\cubictxt} complexes $\mathbb{K} \subset C(n)$ and $\mathbb{K}' \subset C(m)$, an order-preserving map $\varphi : \mathbb{K} \rightarrow \mathbb{K}'$ is called a {\cubictxt} map, if the following conditions are satisfied.
\begin{enumerate*}
\vitem
$\varphi^{-1}(\emptyset) = \{\emptyset\}$,
\hitem
$\forall\,\tau' \in \mathbb{K}' \ \,\forall\,\sigma \in \mathbb{K} \ \ \tau' \prec \varphi(\sigma) \implies \exists\,\tau \prec \sigma$ $\varphi(\tau)=\tau'$.
\end{enumerate*}
In particular, the image of a {\cubictxt} map $\varphi : \mathbb{K} \rightarrow \mathbb{K}'$ is a {\cubictxt} subcomplex of $\mathbb{K}'$.

\begin{prop}
Let $\ast$ be a $0$-{\cubictxt} set.
The following maps are {\cubictxt} maps.
\begin{enumerate}
\item
The trivial map $\phi : \mathbb{K} \to \mathbb{K}(\ast)$ given by $\phi(\emptyset)=\emptyset$ and $\phi(\tau)=\{\ast\}$, \vspace{.5ex}$\tau \in \mathbb{K}\smallsetminus\{\emptyset\}$.
\item
The natural inclusion $\phi : \mathbb{L} \hookrightarrow \mathbb{K}$ of cubic subcomplex $\mathbb{L}$ of $\mathbb{K}$.
\vspace{.5ex}\item
For two {\cubictxt} maps $\phi_{1} : \mathbb{K}(1) \to \mathbb{K}'_{1}$ and $\phi_{2} : \mathbb{K}(2) \to \mathbb{K}'_{2}$, maps
\begin{enumerate}
\vspace{.5ex}\item
$\phi : \mathbb{K}(1) \join \mathbb{K}(2) \to \mathbb{K}'_{1} \join \mathbb{K}'_{2}$ given by $\phi(\tau_{1} \join \tau_{2}) = \phi_{1}(\tau_{1}) \join \phi_{2}(\tau_{2})$ and
\vspace{.5ex}\item
$\psi : \mathbb{K}(1) \times \mathbb{K}(2) \to \mathbb{K}'_{1} \times \mathbb{K}'_{2}$ given by $\psi(\tau_{1} \times \tau_{2}) = \phi_{1}(\tau_{1}) \times \phi_{2}(\tau_{2})$.
\end{enumerate}
\end{enumerate}
\end{prop}

For a {\cubictxt} complex $\mathbb{K} \!\subset\! \Cubic(n)$, $n \!\geq\! 0$, we denote $\mathbb{K}^{q}=\{\,\sigma \!\in\! K \,;\, \text{$\sigma$ is $q$-{\cubictxt}} \,\}$, $q \!\geq\! -1$ and by $\vert{\mathbb{K}}\vert = \underset{\sigma\in \mathbb{K}}{\bigcup}\,\sigma$ the polyhedron in $\real^{n}$ associated to $\mathbb{K}$.
For a {\cubictxt} complexes $\mathbb{K}$ and $\mathbb{L}$, a continuous map $f : \vert{\mathbb{L}}\vert \to \vert{\mathbb{K}}\vert$ is called cubic, if there exists a map $\varphi : \mathbb{L} \to \mathbb{K}$ such that $f\vert_{\tau} : \tau \to \varphi(\tau) \subset \vert{\mathbb{K}}\vert$ for any $\tau \in \mathbb{L}$.
Such a map $f$ is often denoted by $|\varphi| : \vert{\mathbb{L}}\vert \to \vert{\mathbb{K}}\vert$.

\section{{\Smoothtxt} {\cubictxt} complex}

In \diffeology{}, we use the symbol $\I$ for a special diffeological space: let $\I=\real/{\pi_{set}}$ be the diffeological quotient of $\real$, where $\I=[0,1]$ as a set, with a subduction $\pi : \real \to \I$ given by $\pi(t)=\pi_{set}(t) \in [0,1]$, so that we obtain $\dim\I=1$ diffeologically and $\pi{\comp}\pi_{set}=\pi$.
The underlying topology ($D$-Topology in \cite[2.8]{MR3025051}) of $\I$ is the same as $\I_{top} \subset \real$, while $\I \hookrightarrow \real$ can not be an induction.

\begin{thm}\label{cor:subduction-dual}
By Propositions \ref{prop:subduction-dual} and \ref{prop:diffeologicalquotient}, $\pi^{*} : C^{\infty}(\I,X) \rightarrow \Path{X}$ is a diffeomorphism, and so $\I$ represents the functor $\Path{}$.
Moreover we have $\dim\I=1$.
\end{thm}

As is well-known, there is a smooth function $\lambda : \real \to \real$ enjoying 
\begin{enumerate*}
\vitem\label{property:lambda1} $\lambda(t)=0$, $t \le 0$,
\hitem\label{property:lambda2} $\lambda(t)+\lambda(1{-}t)=1$,
\hitem\label{property:lambda3} $\lambda'(t)>0$, $0 < t < 1$.
\end{enumerate*}\par\noindent
Then by (\ref{property:lambda1}) and (\ref{property:lambda2}) above, we have $\lambda(t)=1$, $t \ge 1$ as well and hence $\lambda{\comp}\pi_{set}=\lambda$.
Thus $\lambda$ induces a smooth injection $\hat\lambda : \I \to \real$ satisfying $\lambda=\hat\lambda{\comp}\pi$.

\begin{defn}
We give a diffeology on a $q$-cubic set $\sigma \in C(n)$ by a subduction $\pi_{\sigma} : \real^{q} \to \sigma$, equipped with a smooth injection $\hat\lambda_{\sigma} : \sigma \to \real^{n}$ induced from some smooth map $\lambda_{\sigma} : \real^{q} \to \real^{n}$ such that $\lambda_{\sigma}=\hat\lambda_{\sigma}{\comp}\pi_{\sigma}$, by induction on $q$.
\begin{enumerate}
\item
$0$-cubic set is a one point space which has the trivial diffeology.
\item
Let $\sigma_{1}$ and $\sigma_{2}$ be {\cubicsettxt} with subductions $\pi_{i} : \real^{q_{i}} \to \sigma_{i}$, $i=1, \,2$, equipped with smooth injections $\hat\lambda_{i} : \sigma_{i} \to \real^{n}$ induced from some smooth maps $\lambda_{i} : \real^{q_{i}} \to \real^{n}$ such that $\lambda_{i}=\hat\lambda_{i}{\comp}\pi_{i}$, $i=1,\,2$.
\begin{enumerate}
\item
If $\sigma=\sigma_{1} \join \sigma_{2} \subset \real^{n}$, we have a subduction $\pi_{\sigma} = \phi_{\sigma_{1},\sigma_{2}}{\comp}(\pi_{1} \times \pi \times \pi_{2}) : \real^{q_{1}} \times \real \times \real^{q_{2}} \to \sigma_{1} \join \sigma_{2}$ and a smooth map $\lambda_{\sigma} = \phi_{\sigma_{1},\sigma_{2}}{\comp}(\lambda_{1} \times \lambda \times \lambda_{2}) : \real^{q_{1}} \times \real \times \real^{q_{2}} \to \real^{n}$, 
which induces a smooth injection $\hat\lambda_{\sigma} : \sigma \to \real^{n}$.
\item
If $\sigma = \sigma_{1} \timesab{L_{1}}{L_{2}} \sigma_{2} \subset \real^{n}$, we have a subduction $\pi_{\sigma} = \psi_{\sigma_{1},\sigma_{2}}{\comp}(\pi_{1} \times \pi_{2})$ and a smooth map $\lambda_{\sigma} = \psi_{\sigma_{1},\sigma_{2}}{\comp}(\lambda_{1} \times \lambda_{2}) : \real^{q_{1}} \times \real^{q_{2}} \to \real^{n}$, 
which induces a smooth injection $\hat\lambda_{\sigma} : \sigma \to \real^{n}$.
\end{enumerate}
\end{enumerate}
\end{defn}
For any $q$-cubic set $\sigma$, we clearly have $\dim\sigma=q$ diffeologically, and by Proposition \ref{prop:subduction-dual}, we obtain that $\pi_{\sigma}^{*} : C^{\infty}(\sigma,X) \rightarrow C^{\infty}(\real^{q},X)$ is an induction.

\begin{prop}
$\sigma$ is a smooth neighbourhood deformation retract of $\sigma \join \tau$.
\end{prop}
\begin{proof}
Since there is a subduction $\pi_{\sigma} = \phi_{\sigma_{1},\sigma_{2}}{\comp}(\pi_{1} \times \pi \times \pi_{2}) : \real^{q_{1}} \times \real \times \real^{q_{2}} \to \sigma_{1} \join \sigma_{2}$, we have a deformation $h_{s} : \real^{q_{1}} \times (-\infty,1) \times \real^{q_{2}} \to \real^{q_{1}} \times \real \times \real^{q_{2}}$ given by 
$$
h_{s}(\mathbb{x},t,\mathbb{y})=(\mathbb{x},t{-}\lambda(s),\mathbb{y}),\quad s \in \real,
$$
which is clearly smooth and thus inducing a smooth deformation $\hat{h}_{s} : O \to \sigma \join \tau$, $s \in \real$, where $O=\sigma \join \tau \smallsetminus \tau$.
By definition, it induces a smooth deformation of $O$ relative to $\sigma$, and $\sigma$ is a smooth deformation retract of $O$ a neighbourhood of $\sigma$.
\end{proof}

For a {\cubictxt} complex $\mathbb{K} \subset C(n)$, we introduce a {\smoothtxt} structure on the polyhedron $|\mathbb{K}|$ as $\vert{\mathbb{K}}\vert = \underset{\sigma\in \mathbb{K}}{\colim}\,\sigma$, which is called a {\smoothtxt} cubic polyhedron.
Then by definition, we obtain a smooth injection $\hat\lambda_{\mathbb{K}} : |\mathbb{K}| \to \real^{n}$ by collecting smooth maps $\hat\lambda_{\sigma} : \sigma \to \real^{n}$, $\sigma \in \mathbb{K}$.
Then we instantly see that $\dim{|\mathbb{K}|}=\dim{\mathbb{K}}$ diffeologically.

\begin{rem}
Following the above definition, we obtain a cube $\I^{n} = \left(\real/{\pi_{set}}\right)^{n}$ as a diffeological quotient of $\real^{n}$ which instantly implies $\dim\I^{n} = n$ diffeologically.
The cube $\I^{n}$ is set-theoretically the same as $[0,1]^{n}$, while its diffeology is different from the induced diffeology from $\real^{n}$ which is used in \cite{Haraguchi:2020wq}.
\end{rem}


\section{Associahedra as cubic complexes}

Let us introduce associahedra $K_{n} \subset \real^{n}$, $n \ge 1$ in $\topology{}$ as follows, which is slightly modified from the definition by Stasheff (see \cite{MR0158400}, \cite{MR1000378} or \cite{AX12115741}):
\def\maxmin{\operatorname{m}}
\begin{align*}
&K_{n} = \{\,(t_{1},\ldots,t_{n}) \midvert t_{1}\!=\!0 \le t_{k} \le k{-}1{-}\textstyle\sum_{i=1}^{k-1}t_{i} \ (1 \!<\! k \!<\! n), \ t_{n}\!=\!n{-}1-\textstyle\sum_{i=1}^{n-1}t_{i} \,\},
\end{align*}
or equivalently, we can describe the associahedron as follows.
\begin{align*}
&K'_{n} = \{\,(u_{1},\ldots,u_{n}) \midvert 0\!=\!u_{1} \le u_{2} \le \cdots \le u_{n-1} \le u_{n}\!=\!n{-}1, \ u_{k} \le k{-}1 \ (1 \!<\! k \!<\! n) \,\}.
\end{align*}
Let $H_{1}^{n-2}$\! : \!$x_{1}+{\cdots}+x_{n}\!=\!n{-}1, x_{1}\!=\!0$ be an affine space where $K_{n}$ is a convex body.
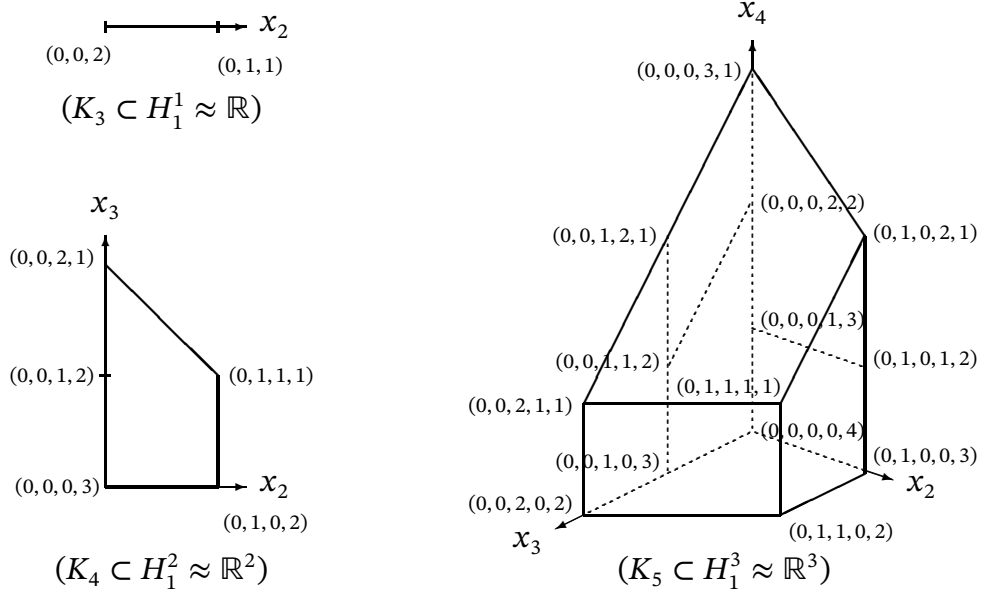
\begin{center}
\setlength\unitlength{.37mm}
\begin{picture}(110,220)(20,-30)
\linethickness{1.0mm}		
\thinlines			
\put(-50, 175)	{\vector(1,0){50}}
\put( 10, 175)	{\makebox(0,0)[cc]{$x_{2}$}}
\thicklines
\put(-10, 173)	{\line(0,1){4}}
\put(-50, 173)	{\line(0,1){4}}
\put(-10, 175)	{\line(-1,0){40}}
\put(-48, 165)	{\makebox(0,0)[r]{\tiny$(0,0,2)$}}
\put(-10, 161)	{\makebox(0,0)[l]{\tiny$(0,1,1)$}}
\put(-30, 145)	{\makebox(0,0)[c]{$(K_{3} \subset H_{1}^{1} \homeo \real)$}}
\thinlines
\put(-50, 50)	{\vector(0,1){50}}
\put(-50,110)	{\makebox(0,0)[cc]{$x_{3}$}}
\put(-50, 10)	{\vector(1,0){50}}
\put( 10, 10)	{\makebox(0,0)[cc]{$x_{2}$}}
\thicklines
\put(-50, 40)	{\line(0,1){50}}
\put(-50, 50)	{\line(0,-1){40}}
\put(-50, 90)	{\line(1,-1){40}}
\put(-10, 10)	{\line(0,1){40}}
\put(-10, 10)	{\line(-1,0){40}}
\put(-52, 50)	{\line(1,0){4}}
\put(-54, 92)	{\makebox(0,0)[r]{\tiny$(0,0,2,1)$}}
\put(-54, 50)	{\makebox(0,0)[r]{\tiny$(0,0,1,2)$}}
\put(-52, 10)	{\makebox(0,0)[r]{\tiny$(0,0,0,3)$}}
\put(-08,-03)	{\makebox(0,0)[l]{\tiny$(0,1,0,2)$}}
\put( -6, 50)	{\makebox(0,0)[l]{\tiny$(0,1,1,1)$}}
\put( -30,-20)	{\makebox(0,0)[c]{$(K_{4} \subset H_{1}^{2} \homeo \real^{2})$}}
\thinlines
\put(180,160)	{\vector(0,1){10}}
\put(180,180)	{\makebox(0,0)[cc]{$x_{4}$}}
\put(120,00)	{\vector(-2,-1){10}}
\put(100,-10)	{\makebox(0,0)[cc]{$x_{3}$}}
\put(220,16)	{\vector(3,-1){10}}
\put(240,10)	{\makebox(0,0)[cc]{$x_{2}$}}
\curvedashes[.3mm]{1,1,1}
\put(180,30)	{\curve(0,0 , 0,130)}
\put(180,30)	{\curve(0,0 , -60,-30)}
\put(180,30)	{\curve(0,0 , 40,-14)}
\put(150,15)	{\curve(0,0 , 0,40)}
\put(150,53)	{\curve(0,0 , 30,60)}
\put(150,55)	{\curve(0,0 , 0,45)}
\put(220,53)	{\curve(0,0 , -40,14)}
\put(170,-20)	{\makebox(0,0)[c]{$(K_{5} \subset H_{1}^{3} \homeo \real^{3})$}}
\thicklines
\put(120,00)	{\line(0,1){40}}
\put(120,40)	{\line(1,2){60}}
\put(220,100)	{\line(-2,3){40}}
\put(190,40)	{\line(-1,0){70}}
\put(190,00)	{\line(-1,0){70}}
\put(220,15)	{\line(0,1){85}}
\put(220,15)	{\line(-2,-1){30}}
\put(190,00)	{\line(0,1){40}}
\put(190,40)	{\line(1,2){30}}
\put(190,46)		{\makebox(0,0)[rc]{\tiny$(0,1,1,1,1)$}}
\put(147, 55)		{\makebox(0,0)[rc]{\tiny$(0,0,1,1,2)$}}
\put(147, 19)		{\makebox(0,0)[rc]{\tiny$(0,0,1,0,3)$}}
\put(176,159)		{\makebox(0,0)[rc]{\tiny$(0,0,0,3,1)$}}
\put(146,100)		{\makebox(0,0)[rc]{\tiny$(0,0,1,2,1)$}}
\put(116, 39)		{\makebox(0,0)[rc]{\tiny$(0,0,2,1,1)$}}
\put(116,  4)		{\makebox(0,0)[rc]{\tiny$(0,0,2,0,2)$}}
\put(183, 30)		{\makebox(0,0)[lc]{\tiny$(0,0,0,0,4)$}}
\put(183, 70)		{\makebox(0,0)[lc]{\tiny$(0,0,0,1,3)$}}
\put(183,112)		{\makebox(0,0)[lc]{\tiny$(0,0,0,2,2)$}}
\put(223,101)		{\makebox(0,0)[lc]{\tiny$(0,1,0,2,1)$}}
\put(223, 56)		{\makebox(0,0)[lc]{\tiny$(0,1,0,1,2)$}}
\put(223, 21)		{\makebox(0,0)[lc]{\tiny$(0,1,0,0,3)$}}
\put(193, -6)		{\makebox(0,0)[lc]{\tiny$(0,1,1,0,2)$}}
\end{picture}
\end{center}

Let $A(n)=\{\,(k,r,s) \in \numeric \midvert 1 \!\le\! k \!\le\! r, \ 2\!\le\!s\!=\!n{-}r{+}1\!\le\!n{-}1\,\}$.
Then the boundary of $K_{n}$ is the union of faces corresponding to elements in $A(n)$, given as follows.
\begin{align*}&
L_{k}(r,s) = \{\,(t_{1},\ldots,t_{n}) \!\in\! K_{n} \midvert (t_{k},\ldots,t_{k+s-2},t) \!\in\! K_{s}, \,t_{k+s-1} \!\ge\! t \!=\! s{-}1-\textstyle\sum_{i=k}^{k+s-2}t_{k+i-1}\,\}.
\end{align*}
\begin{center}
\setlength\unitlength{.37mm}
\begin{picture}(105,205)(20,-30)
\put(-70,00)	{\line(0,1){40}}
\put(-70,40)	{\line(1,2){60}}
\put( 30,100)	{\line(-2,3){40}}
\put(  0,40)	{\line(-1,0){70}}
\put(  0,00)	{\line(-1,0){70}}
\put( 30,15)	{\line(0,1){85}}
\put( 30,15)	{\line(-2,-1){30}}
\put(  0,00)	{\line(0,1){40}}
\put(  0,40)	{\line(1,2){30}}
\put(-41,101)	{\curve(0,0 , 3,-1.5)}
\put( 29, 52)	{\curve(0,0 , 2,0)}
\put(180,30)	{\line(0,1){130}}
\put(180,30)	{\line(-2,-1){60}}
\put(180,30)	{\line(3,-1){40}}
\put(150,15)	{\line(0,1){40}}
\put(150,53)	{\line(1,2){30}}
\put(150,55)	{\line(0,1){45}}
\put(220,53)	{\line(-3,1){40}}
\put(70,-20)	{\makebox(0,0)[c]{$( \partial K_{5} \homeo S^{2})$}}
\put( 15,26)	{\makebox(0,0)[c]{\tiny$L_{1}(4,2)$}}
\put(-35,20)	{\makebox(0,0)[c]{\tiny$L_{1}(3,3)$}}
\put(-15,80)	{\makebox(0,0)[c]{\tiny$L_{1}(2,4)$}}
\put(165,50)	{\makebox(0,0)[c]{\tiny$L_{2}(2,4)$}}
\put(166,106)	{\makebox(0,0)[c]{\tiny$L_{2}(3,3)$}}
\put(135,36)	{\makebox(0,0)[c]{\tiny$L_{2}(4,2)$}}
\put(200,43)	{\makebox(0,0)[c]{\tiny$L_{3}(3,3)$}}
\put(200,90)	{\makebox(0,0)[c]{\tiny$L_{3}(4,2)$}}
\put(180,15)	{\makebox(0,0)[c]{\tiny$L_{4}(4,2)$}}
\put(-15,159)		{\makebox(0,0)[rc]{\tiny$(0,0,0,3,1)$}}
\put(-44,101)		{\makebox(0,0)[rc]{\tiny$(0,0,1,2,1)$}}
\put(-74, 40)		{\makebox(0,0)[rc]{\tiny$(0,0,2,1,1)$}}
\put(-74, -2)		{\makebox(0,0)[rc]{\tiny$(0,0,2,0,2)$}}
\put( 03, 40)		{\makebox(0,0)[lc]{\tiny$(0,1,1,1,1)$}}
\put( 33,101)		{\makebox(0,0)[lc]{\tiny$(0,1,0,2,1)$}}
\put( 33, 53)		{\makebox(0,0)[lc]{\tiny$(0,1,0,1,2)$}}
\put( 33, 16)		{\makebox(0,0)[lc]{\tiny$(0,1,0,0,3)$}}
\put( 01, -5)		{\makebox(0,0)[lc]{\tiny$(0,1,1,0,2)$}}
\put(147, 55)		{\makebox(0,0)[rc]{\tiny$(0,0,1,1,2)$}}
\put(147, 19)		{\makebox(0,0)[rc]{\tiny$(0,0,1,0,3)$}}
\put(176,159)		{\makebox(0,0)[rc]{\tiny$(0,0,0,3,1)$}}
\put(146,101)		{\makebox(0,0)[rc]{\tiny$(0,0,1,2,1)$}}
\put(116, 39)		{\makebox(0,0)[rc]{\tiny$(0,0,2,1,1)$}}
\put(116, -2)		{\makebox(0,0)[rc]{\tiny$(0,0,2,0,2)$}}
\put(183, 30)		{\makebox(0,0)[lc]{\tiny$(0,0,0,0,4)$}}
\put(183, 70)		{\makebox(0,0)[lc]{\tiny$(0,0,0,1,3)$}}
\put(183,112)		{\makebox(0,0)[lc]{\tiny$(0,0,0,2,2)$}}
\put(223,101)		{\makebox(0,0)[lc]{\tiny$(0,1,0,2,1)$}}
\put(223, 53)		{\makebox(0,0)[lc]{\tiny$(0,1,0,1,2)$}}
\put(223, 16)		{\makebox(0,0)[lc]{\tiny$(0,1,0,0,3)$}}
\put(191, -5)		{\makebox(0,0)[lc]{\tiny$(0,1,1,0,2)$}}
\put(120,00)	{\line(0,1){40}}
\put(120,40)	{\line(1,2){60}}
\put(220,100)	{\line(-2,3){40}}
\put(190,00)	{\line(-1,0){70}}
\put(220,15)	{\line(0,1){85}}
\put(220,15)	{\line(-2,-1){30}}
\end{picture}
\end{center}
Following Stasheff \cite{MR0158400} (see also \cite{MR1000378} or \cite{AX12115741}), we introduce face operators $\partial_{k} : \real^{r} \times \real^{s} \to \real^{n}$, $r+s=n+1$, $1 \le k \le r$, as the following linear maps.\vspace{1ex}
$$\partial_{k}((t_{1},\ldots,t_{r}),(u_{1},\ldots,u_{s})) = 
\begin{cases}\,
(u_{1},\ldots,u_{s-1},u_{s}{+}t_{1},t_{2},\ldots,t_{r}),&\!\! k\!=\!1,
\\[.5ex]\,
(t_{1},\ldots,t_{k-1},u_{1},\ldots,u_{s-1},u_{s}{+}t_{k},t_{k+1},\ldots,t_{r}),&\!\! 2\!\le\!k\!\le\!r,
\end{cases}
$$\vskip1ex
If we restrict $\partial_{k}$ to $K_{r} \times K_{s}$, then we obtain $\partial_{k} : K_{r} \times K_{s} \homeo L_{k}(r,s) \subset K_{n} \subset \real^{n}$.
Now we choose an interior point of $K_{n}$, as $\mathbb{b}_{n}=(0,\fracinline1/2,\ldots,\fracinline1/2,\fracinline{n}/2) \in K_{n}$.
Then we see that $K_{n}$ and $L_{k}(r,s)$ are characterised by the following two conditions.\vspace{.5ex}
\begin{enumerate*}
\vitem
$\partial_{k} : K_{r} \times K_{s} \xrightarrow{\homeo} L_{k}(r,s)$, $(k,r,s) \in A(n)$,\vspace{.5ex}
\vitem
$K_{n} = \bigcup_{(k,r,s) \in A(n)} L_{k}(r,s) \join \mathbb{b}_{n}$.
\end{enumerate*}
\begin{expl}
$K_{2} = \{\mathbb{b}_{2}\}$, $K_{3} = L_{2}(2,2) \join \{\mathbb{b}_{3}\} \join L_{1}(2,2)$, and $K_{4} = L_{2}(2,3) \join \mathbb{b}_{4} \cup L_{2}(3,2) \join \mathbb{b}_{4} \cup L_{3}(3,2) \join \mathbb{b}_{4} \cup L_{1}(2,3) \join \mathbb{b}_{4} \cup L_{1}(3,2) \join \mathbb{b}_{4}$.
\end{expl}
Firstly, $K_{n}$ is the realisation of a {\cubicsettxt}, namely $\mathbb{K}(n)$, which begins with 
\begin{enumerate}
\item
$\mathbb{K}(1)=\{\emptyset\}$.
\end{enumerate}
Secondly, we define {\cubictxt} complexes $\mathbb{L}_{k}(r,s)$, $(k,r,s) \in A(n)$, assuming that $\mathbb{K}(r)$ and $\mathbb{K}(s)$ are given: 
let us denote $\mathbb{0}_{n}=(0,\ldots,0) \in \real^{n}$ and $H_{n}=\{(0,x_{2},\ldots,x_{n}) \in \real^{n} \midvert x_{2}+\cdots+x_{n}=0\}$.
Then we have the following two linear subspaces of $\real^{n}$:
\begin{align*}&
V_{1}=\partial_{k}(H_{r} \times \{\mathbb{0}_{s}\}) 
\quad\text{and}
\\&
V_{2}=\partial_{k}(\{\mathbb{0}_{r}\} \times H_{s}) 
\quad\text{with}\quad V_{1} \cap V_{2} = \{\mathbb{0}_{n}\}.
\end{align*}
Let $\mathbb{a}=\partial_{k}(\mathbb{b}_{r},\mathbb{b}_{s}) \in \real^{n}$.
Since $K_{n} \subset H_{n}+\mathbb{b}_{n}$, we obtain 
\begin{align*}&
L_{1} = V_{1}+\mathbb{a} = \partial_{k}((H_{r}\!+\!\mathbb{b}_{r}) \times \{\mathbb{b}_{s}\}) \supset \partial_{k}(K_{r} \times \{\mathbb{b}_{s}\}) \homeo K_{r},
\\&
L_{2} = V_{2}+\mathbb{a} = \partial_{k}(\{\mathbb{b}_{r}\} \times (H_{s}\!+\!\mathbb{b}_{s})) \supset \partial_{k}(\{\mathbb{b}_{r}\} \times K_{s}) \homeo K_{s}.
\end{align*}
Then we clearly have that $L_{1} \cap L_{2} = \{\mathbb{a}\}$.
Hence, when $\mathbb{K}(r)$ and $\mathbb{K}(s)$ have already been defined, we must obtain the following.
\begin{enumerate}
\addtocounter{enumi}{1}
\item
$\mathbb{L}_{k}(r,s) = \partial_{k}(\mathbb{K}(r) \times \{\mathbb{b}_{s}\})\,\,{}_{L_{1}\!\!}\times_{L_{2}} \partial_{k}(\{\mathbb{b}_{r}\} \times \mathbb{K}(s))$, $(k,r,s) \in A(n)$.
\end{enumerate}\vspace{.5ex}

Thirdly, we define a {\cubictxt} complex $\mathbb{K}(n)$: 
let $V'_{1}=V_{1}+V_{2}$, $V'_{2}=\{\mathbb{0}_{n}\}$, $\mathbb{a}'_{1}=\mathbb{a}$ and $\mathbb{a}'_{2}=\mathbb{b}_{n}$.
Then ${L}_{k}(r,s) \subset V'_{1}+\mathbb{a}'_{1}$, $\mathbb{b}_{n}\in V'_{2}+\mathbb{a}'_{2}$ and $V'_{1} \cap V'_{2} = \emptyset$.
Hence, when $\mathbb{L}_{k}(r,s)$, $(k,r,s) \in A(n)$, have already been defined, we must obtain the following.
\begin{enumerate}
\addtocounter{enumi}{2}
\item
$\mathbb{K}(n) = \bigcup_{(k,r,s) \in A(n)} \mathbb{L}_{k}(r,s) \join \mathbb{b}_{n}$.
\end{enumerate}
Thus $\mathbb{K}(n)$, $n \ge 1$ can inductively be defined by the above formulas (1), (2) and (3).
Then we can easily see that there is a natural homeomorphism from $|\mathbb{K}(n)|$ to $K_{n}$.

By slightly modifying the definition in Stasheff \cite{MR0158400}, we obtain the following degeneracy operators $s_{j} : K_{n} \to K_{n-1}$, $1 \!\le\! j \!\le\! n$, $n \!\ge\! 2$ (see \cite{MR1000378} or \cite{AX12115741}).
\begin{align*}&
s_{j}((1{-}t){\cdot}\partial_{k}(\rho,\sigma) + t{\cdot}\mathbb{b}_{n}) = (1{-}t){\cdot}s_{j}\comp\partial_{k}(\rho,\sigma) + t{\cdot}\mathbb{b}_{n-1},
\end{align*}
where $s_{j}{\comp}\partial_{k} : K_{r} \times K_{s} \to K_{n-1}$ is given by\vspace{1ex}
$$
s_{j}\comp\partial_{k}(\rho,\sigma) = 
\begin{cases}\partial_{k-1}(r{-}1,s){\comp}(s_{j}\rho \times \sigma),& j \!<\! k, \ r\!>\!2,
\\[-.0ex]
\sigma,& j\!=\!1, \,k\!=\!2, \,r\!=\!2,
\\[.5ex]\partial_{k}(r,s{-}1){\comp}(\rho \times s_{j-k+1}\sigma),& k \!\le\! j \!<\! k{+}s, \ r \!<\! n{-}1,
\\[-.0ex]
\rho,& k\!\le\!j\!\le\!k{+}1, \ r\!=\!n{-}1,
\\[.5ex]\partial_{k-1}(r{-}1,s){\comp}(s_{j-s+1}\rho \times \sigma),& k{+}s \!\le\! j \!\le\! n, \ r\!>\!2,
\\[-.0ex]
\sigma,& j\!=n, \,k\!=\!1, \,r\!=\!2.
\end{cases}
$$\vskip1ex\noindent
Thus we may suppose that $s_{j} : K_{n} \to K_{n-1}$, $1 \le j \le n$, $n \ge 2$ is a realisation of a cubic map, which is denoted again by $s_{j} : \mathbb{K}(n) \to \mathbb{K}(n-1)$.

\begin{rem}
J.~L.~Loday \cite{MR2342835,MR3221534} gave a nice realisation of an associahedron and its triangulation using Tamari ordering on vertices, which could give a natural smooth structure on such triangulation in terms of trees.
\end{rem}

\section{$A_{\infty}$-form in $\diffeology$}

In this section, let us concentrate on our cubic complex of assosiahedra.
The following is obtained by induction on $n$.

\begin{thm}
$\partial_{k} : \mathbb{K}(r) \times \mathbb{K}(s) \to \mathbb{K}(n)$, $r+s=n+1$, $1 \le k \le r$ and $s_{j} : \mathbb{K}(n) \to \mathbb{K}(n{-}1)$, $1 \le j \le n \ge 2$ are smooth cubic maps.
\end{thm}

We now state the smooth version of a strict unital $A_{\infty}$-space: let $G$ be a diffeological space with a base point $e \in G$, which is called a unit of $G$.
\begin{defn}[Stasheff \cite{MR0158400}]\label{defn:stasheff63}
$G$ is called a smooth (strict unital) $A_{\infty}$-space, if there is a series of smooth maps $\{ M(n) : |\mathbb{K}(n)| \times G^{n} \to G \}_{n \ge 2}$ ($A_{\infty}$-form) satisfying 
\begin{enumerate}\vspace{.5ex}
\item
$M(n)(\partial_{k}(\rho,\sigma);g_{1},\ldots,g_{n})=M(r)(\rho;g_{1},\ldots,M(s)(\sigma;g_{k},\ldots,g_{k+s-1}),\ldots,g_{n})$.
\vspace{.5ex}\item
$M(n)(\tau;g_{1},\ldots,g_{j-1},e,g_{j+1},\ldots,g_{n})=M(n{-}1)(s_{j}(\tau);g_{1},\ldots,g_{j-1},g_{j+1},\ldots,g_{n})$.
\end{enumerate}
\end{defn}

An $A_{\infty}$-space allowing homotopy unit was introduced in \cite{MR0270372} as in the following definition, which shall be referred as h-unital $A_{\infty}$-space in this article.

\begin{defn}[Stasheff \cite{MR0270372}]\label{defn:stasheff70}
$G$ is called a smooth h-unital $A_{\infty}$-space, if there is a series of smooth maps $\{ M(n) : |\mathbb{K}(n)| \times G^{n} \to G \}_{n \ge 2}$ ($A_{\infty}$-form) satisfying 
\begin{enumerate}\vspace{.5ex}
\item[(1)]
$M(n)(\partial_{k}(\rho,\sigma);g_{1},\ldots,g_{n})=M(r)(\rho;g_{1},\ldots,M(s)(\sigma;g_{k},\ldots,g_{k+s-1}),\ldots,g_{n})$.
\vspace{.5ex}\item[(2')\!]\,%
$M_{2}(e,g) \simeq M_{2}(g,e) \simeq g$ the identity.
\end{enumerate}
\end{defn}
By \cite[Theorem 1.4]{AX12115741}, these definitions are not the same but equivalent up to homotopy.
In view of \cite{AX12115741}, we give two notions of internal precategories: 
\begin{defn}
A pair of diffeological spaces $G=(G,X)$ is called an internal precategory, if it is equipped with three smooth (structure) maps
\begin{align*}&
\sigma : G \to X,
\quad
\tau : G \to X,
\quad
\iota : X \to G
\end{align*}
with relations $\sigma{\comp}\iota=\tau{\comp}\iota=\id$.
\end{defn}
For an internal precategory $G=(G,X)$, we define $\{G^{n}; \sigma_{n}, \tau_{n} : G^{n} \to X\}_{n \ge 1}$ by induction:
\begin{enumerate*}
\vitem $G^{1}=G$, \,$\sigma_{1}=\sigma$ and $\tau_{1}=\tau$
\vitem[(n{+}1)] $G^{n+1}=G^{n} \times_{X} G$, \,$\sigma_{n+1}(\chi,h)=\sigma_{n}(\chi)$ and $\tau_{n+1}(\chi,h)=\tau(h)$,
\end{enumerate*}\par\noindent
where $G^{n} \times_{X} G$ is the pullback of $\tau_{n}$ and $\sigma$, that is, $G^{n} \times_{X} G = \{\,(\chi,h) \in G^{n} \times G \midvert \tau_{n}(\chi)\!=\!\sigma(h)\,\}$ is a subspace of $G^{n} \times G$.
\begin{defn}\label{defn:internal-higher-associativity}
An internal precategory $G=(G,X)$ in \diffeology{} is called an internal (strict unital) smooth $A_{\infty}$-category,\vspace{.5ex} if there is a series of smooth maps $\{ M(n) : |\mathbb{K}(n)| \times G^{n} \to G \}_{n \ge 2}$ ($A_{\infty}$-form) satisfying the conditions (0), (1) and (2):
\begin{enumerate}
\vskip1ex
\item[(0)]
$\sigma{\comp}M(n)(\tau;g_{1},\dots,g_{n})=\sigma(g_{1})$ and $\tau{\comp}M(n)(\tau;g_{1},\dots,g_{n})=\tau(g_{n})$.
\vspace{.5ex}\item%
$M(n){\comp}(\partial_{k} \times \id^{n}) = M(r){\comp}(\id \times \id^{k-1} \times_{X} M(s) \times_{X} \id^{r-k}){\comp}(\id \times T_{k}) : |\mathbb{K}(r)| \times |\mathbb{K}(s)| \times G^{n} \to G$,\!\!\!\!\!\!\!\!\!\!
\begin{equation*}
\begin{diagram}
\node{|\mathbb{K}(r)| \times |\mathbb{K}(s)| \times G^{n}}
\arrow[2]{e,b}{\partial_{k} \times \id^{n}}
\arrow{s,l}{\id \times T_{k}}
\node{}
\node{|\mathbb{K}(n)| \times G^{n}}
\arrow[2]{s,r}{M(n)}
\\
\node{|\mathbb{K}(r)| \times G^{k-1} \times_{X} (|\mathbb{K}(s)| \times G^{s}) \times_{X} G^{r-k}}
\arrow{s,l}{\id \times \id \times_{X} M(s) \times_{X} \id}
\\
\node{|\mathbb{K}(r)| \times G^{r}}
\arrow[2]{e,t}{M(r)}
\node{}
\node{G,}
\end{diagram}
\end{equation*}
where $T_{k} : |\mathbb{K}(s)| \times G^{n} = |\mathbb{K}(s)| \times G^{k-1} \times_{X} G^{s} \times_{X} G^{r-k} \to G^{k-1} \times_{X} (|\mathbb{K}(s)| \times G^{s}) \times_{X} G^{r-k}$ is given by $T_{k}(a,\mathbb{x},\mathbb{y},\mathbb{z}) = (\mathbb{x},a,\mathbb{y},\mathbb{z})$.
\vspace{.5ex}\item%
$M(n){\comp}(\id \times \id^{j-1} \times_{X} \iota \times_{X} \id^{n-j}) = M(n{-}1) {\comp} (s_{j} \times \id^{j-1} \times_{X} \id^{n-j}) : |\mathbb{K}(n)| \times G^{n-1} \to G$.
\begin{equation*}
\begin{diagram}
\node{|\mathbb{K}(n)| \times G^{j-1} \times_{X} G^{n-j}}
\arrow{e,J}
\arrow{s,l}{s_{j} \times \id \times_{X} \id}
\node{|\mathbb{K}(n)| \times G^{n}}
\arrow{s,r}{M(n)}
\\
\node{|\mathbb{K}(n{-}1)| \times G^{j-1} \times_{X} G^{n-j}}
\arrow{e,b}{M(n{-}1)}
\node{G,}
\end{diagram}
\end{equation*}
\end{enumerate}
\end{defn}

\begin{defn}\label{defn:internal-higher-associativity'}
We call $G=(G,X)$ an internal smooth h-unital $A_{\infty}$-category in \diffeology{}, if there is a series of smooth maps $\{M(n) : |\mathbb{K}(n)| \times G^{n} \to G\}_{n \ge 2}$ ($A_{\infty}$-form) satisfying the conditions (0) and (1) in Definition \ref{defn:internal-higher-associativity} and (2') below.
\begin{enumerate}
\item[(2')\!]\,%
$M(2)|_{\{\mathbb{b}_{2}\} \times \{e\} \times G} \simeq M(1){\comp}(s_{1} \times \id)$ and $M(2)|_{\{\mathbb{b}_{2}\} \times G \times \{e\}} \simeq M(1){\comp}(s_{2} \times \id)$ the identity.
\end{enumerate}
\end{defn}

\section{Path space with usual concatenation}

In this section, we work in \diffeology{} making some additional assumptions on the diffeology $\mathcal{D}(X)$ of a diffeological space $X$.
Let us denote by $\mathcal{F}(X) = C^{\infty}(X,\real)$ the set of {\smoothtxt} functions on $X$ to $\real$.
Then we define a superset $\mathcal{D}'(X)$ of $\mathcal{D}(X)$ as the set of parametrisations $P$ on $X$ satisfying that $\varphi{\comp}P$ is {\smoothtxt} for any $\varphi \in \mathcal{F}(X)$.
\begin{defn}[J.~Watts \cite{MR3153238} (see also \cite{MR3025051})]\label{defn:reflexive}
A diffeological space $X$ is said to be reflexive, if $\mathcal{D}(X) = \mathcal{D}'(X)$.
\end{defn}

\begin{thm}[\protect{\cite[Exercise 79]{MR3025051}}]\label{thm:reflexivesmoothCW}
A manifold is reflexive in $\diffeology{}$.
\end{thm}

Let us introduce the following notion for a point $\mathbb{a}$ in $X$.

\begin{defn}\label{defn:reflexivesmoothCW}
A diffeological space $X$ is said to be reflexive at $\mathbb{a} \in X$, if there is a $D$-open neighbourhood $U$ of $\mathbb{a}$, which is reflexive as a diffeological space.
\end{defn}

\begin{expl}\label{example:reflexive}
A manifold is reflexive at any point, and hence a smooth CW complex of finite dimension is reflexive at an interior point of a top cell.
\end{expl}

\begin{expl}\label{example:non-reflexive}
Let $f : \real \to \I$ be a map given by
$$
f(t)=\pi(\max\{0,\sqrt{t}\}),
$$
which is not smooth at $t\!=\!0$.
In fact, if $f$ is smooth, $f$ can be expressed as $f=\pi{\comp}\phi$ near $t\!=\!0$ by a smooth map $\phi : \real \to \real$.
Then we have $\phi(t)=\sqrt{t}$ for small $t>0$, and hence $\phi'(0)=\underset{t \to +0}\lim \phi'(t)=+\infty$.
It contradicts to the smoothness of $\phi$ at $t=0$.
On the other hand, for any smooth map $g : \I \to \real$, the composition $\psi=g{\comp}\pi : \real \to \real$ is smooth on $\real$, and constant on $(-\infty,0]$.
Thus we have 
$$
\underset{t \to +0}\lim\,\psi^{(r)}(t)=\psi^{(r)}(0)=\underset{t \to -0}\lim\,\psi^{(r)}(t)=0\quad\text{for all $r \!\ge\! 1$.}
$$\vskip-1ex\noindent
By applying L'H\^opital's rule many times, we obtain that $\underset{t \to +0}\lim\,\psi^{(r)}(t)/t^{n}=0$ for all $r, \,n \ge 1$.
Then by induction, we can express $(\psi{\comp}f)^{(r)}(t)$ as the following form:
$$
(\psi{\comp}f)^{(r)}(t) = \underset{j=0}{\overset{r}{\sum}}\,P_{\!r,j}(1/\sqrt{t}){\cdot}\psi^{(j)}(\sqrt{t}), \,t>0,\quad\text{for all $r > 1$,}
$$\vskip-1ex\noindent
where $P_{\!r,j}(x)$ is a polynomial on $x$.
Again by applying L'H\^opital's rule, we obtain that $(\psi{\comp}f)^{(r)}(0)$ exists and equals to $\underset{t \to 0}\lim\,(\psi{\comp}f)^{(r)}(t)=0$ for all $r \!\ge\! 1$, and hence $\psi{\comp}f$ is smooth at $t\!=\!0$.
Thus $f \in \mathcal{D}'(\I)$ while $f \not\in \mathcal{D}(\I)$.
So, $\I$ is not reflexive (at $t\!=\!0$).
In contrast, $\I$ is reflexive at any point $t \in (0,1) \subset \I$.
\end{expl}

Let us recall the following diffeological subspace of $\Paths(X)$.
$$
\Path{X}=\{\, u \in \Paths(X) \midvert u{\comp}\pi_{set}=u\,\} \cong C^{\infty}(\I,X).
$$
From now on, we often identify $\Path{X}$ with $C^{\infty}(\I,X)$ without mentioning $\pi^{*}$.

\begin{prop}\label{prop:higher-derivaltive}
For $u \in \Path{X} = C^{\infty}(\I,X)$ and $\varphi \in C^{\infty}(X,\real)$,
 we obtain
\begin{equation*}
\nder{(\varphi{\comp}u)}by{t}times{n} = 0 \ \ \text{on} \ \ (-\infty,0] \cup [1,\infty) \ \ \text{for all} \ \ n > 0.
\end{equation*}
\end{prop}
\begin{proof}
Since $u$ is constant on $(-\infty,0)$ and on $(1,\infty)$, so is $\varphi{\comp}u$ and we obtain 
\begin{equation*}
\der{(\varphi{\comp}u)}by{t} = 0 \ \ \text{on} \ \ (-\infty,0) \cup (1,\infty).
\end{equation*}
Hence by iterating differentiations, we obtain
\begin{equation*}
\nder{(\varphi{\comp}u)}by{t}times{n} = 0 \ \ \text{on} \ \ (-\infty,0) \cup (1,\infty) \ \ \text{for all} \ \ n > 0.
\end{equation*}
Because all the derivatives are continuous, we obtain the propostion.
\end{proof}

Since $\Path{X} = C^{\infty}(\I,X)$ is the left adjoint of the product functor with $\I$, $\Path{X}$ is an internal precategory with the following smooth structure maps. 
\begin{align*}&
\sigma=\sigma_{X} : \Path{X} \to X \iff \sigma_{X}(u)=u(0),\quad u \in \Path{X},
\\
&\tau=\tau_{X} : \Path{X} \to X \iff \tau_{X}(u)=u(1),\quad u \in \Path{X},
\\
&\iota=\iota_{X} : X \to \Path{X} \iff \iota_{X}(\mathbb{x})(t)=\mathbb{x},\quad \mathbb{x} \in X \ \& \ t \in \real.
\intertext{The internal precategory $\Path{X}$ is equipped with the following structure map $\mu=\mu_{X}$ called a concatenation, which could fail to be well-defined in general:}
&\mu=\mu_{X} : \Path{X} \times_{X} \Path{X} \dashrightarrow  \Path{X}
\\[-.0ex]&\qquad\qquad
\iff \mu_{X}(u,v)(t) = 
\begin{cases}\,
u(2t),		& t\le\fracinline1/2,
\\\,
v(2t{-}1),	& t\ge\fracinline1/2,
\end{cases} \ \ t \in \real,
\end{align*}
where $\Path{X} \times_{X} \Path{X}$ denotes the pullback of $\tau : \Path{X} \to X$ and $\sigma : \Path{X} \to X$:
$$
\Path{X} \times_{X} \Path{X} = \{\,(u,v) \in \Path{X} \times \Path{X} \midvert \tau(u)=\sigma(v)\,\}.
$$
\begin{defn}%
We define the following subspaces of $\Path{X} \subset C^{\infty}(\real,X)$.
\begin{enumerate}
\item
$\Path{X;\mathbb{a},\mathbb{b}} = \{\, u \in \Path{X} \midvert u(0)=\mathbb{a} \ \& \ u(1)=\mathbb{b} \,\}$.
\item
$\Loop{X,\mathbb{a}} = \Path{X;\mathbb{a},\mathbb{a}}$ and $\Loop{X} = \Path{X;\mathbb{\ast},\mathbb{\ast}}$.
\end{enumerate}
\end{defn}

Let $\mathbb{a}, \mathbb{b}, \mathbb{c} \in X$.
In the remainder of this section, we assume that $X$ is reflexive at $\mathbb{b} \in X$.
For any two paths $u \in \Path{X;\mathbb{a},\mathbb{b}}$ and $v \in \Path{X;\mathbb{b},\mathbb{c}}$, and any function $\varphi \in \mathcal{F}(X)$, both $u$ and $v$ are smooth, and $\varphi{\comp}u$ and $\varphi{\comp}v$ are {\smoothtxt} functions in the ordinary sense. 
The function $\varphi{\comp}\mu(u,v)$ can be described as follows: 
\begin{equation}
\varphi{\comp}\mu(u,v)(t) = \mu_{\real}(\varphi{\comp}u,\varphi{\comp}v)(t) = \begin{cases}\,
\varphi{\comp}u(2t),		& t\le\fracinline1/2,
\\\,
\varphi{\comp}v(2t{-}1),	& t\ge\fracinline1/2.
\end{cases} \ \ t \in \real.
\tag{A}\label{eq:concat-path}
\end{equation}
Then, $\varphi{\comp}\mu(u,v) = \mu_{\real}(\varphi{\comp}u,\varphi{\comp}v)$ is smooth:
by Proposition \ref{prop:higher-derivaltive}, we obtain
\begin{equation*}
\nder{(\varphi{\comp}u)}by{t}times{n} = 0 \ \ \text{and} \ \ \nder{(\varphi{\comp}v)}by{t}times{n} = 0 \ \ \text{on} \ \ (-\infty,0] \cup [1,\infty) \ \ \text{for all} \ \ n > 0.
\end{equation*}
So we define $\varphi_{1}, \varphi_{2} \in \Path{X}$ by $\varphi_{1}(t) =\varphi{\comp}u(2t)$ and $\varphi_{2}(t) =\varphi{\comp}v(2t{-}1)$, $t \in \real$ to obtain $\varphi_{1} = \varphi{\comp}\mu(u,k(\mathbb{b})) = \mu_{\real}(\varphi{\comp}u,k(\varphi(\mathbb{b})))$ and $\varphi_{2} = \varphi{\comp}\mu(k(\mathbb{b}),v) = \mu_{\real}(k(\varphi(\mathbb{b})),\varphi{\comp}v)$, where $k(\mathbb{b})$ and $k(\varphi(\mathbb{b}))$ are the constant maps at $\mathbb{b} \in X$ and $\varphi(\mathbb{b}) \in \real$, respectively.
Since $\varphi_{1}$ and $\varphi_{2}$ are compositions of smooth maps, they are smooth satisfying 
\begin{align*}&
\nder{\!\varphi_{1}}by{t}times{n} = 0 \ \ \text{on} \ \ (-\infty,0] \cup [\fracinline1/2,\infty) \ \ \text{for all} \ \ n > 0, \ \ \text{and}
\\[1ex]&
\nder{\!\varphi_{2}}by{t}times{n} = 0 \ \ \text{on} \ \ (-\infty,\fracinline1/2] \cup [1,\infty) \ \ \text{for all} \ \ n > 0.
\end{align*}
Since $\varphi{\comp}\mu(u,v)(t) = \mu_{\real}(\varphi{\comp}u,\varphi{\comp}v)(t) = \varphi_{1}(t) + \varphi_{2}(t) - \varphi(\mathbb{b})$, $t \!\in\! \real$ by the equation (\ref{eq:concat-path}), the composition $\varphi{\comp}\mu(u,v) = \mu_{\real}(\varphi{\comp}u,\varphi{\comp}v)$ is also a smooth function satisfying 
\begin{equation*}
\nder{(\varphi{\comp}\mu(u,v))}by{t}times{n} = 0 \ \ \text{on} \ \ (-\infty,0] \cup \{\fracinline1/2\} \cup [1,\infty) \ \ \text{for all} \ \ n > 0.
\end{equation*}
Here, let us recall that $\mu(u,v)$ $:$ $\real \to X$ is smooth on $\real \smallsetminus \{\fracinline1/2\}$ and that there is a $D$-open neighbourhood $U \subset X$ of $\mathbb{b}$ such that $U$ is reflexive, since $X$ is reflexive at the point $\mathbb{b}$. 
Since $\varphi{\comp}\mu(u,v)$ is smooth for any smooth function $\varphi : X \to \real$, we obtain that $\mu(u,v)$ is smooth on the open set $\mu(u,v)^{-1}(U) \ni \fracinline1/2$.
It means that $\mu(u,v)$ is a plot, and $\mu$ is well-defined.

\begin{thm}\label{thm:concatenation}
Let $X$ be a diffeological space and reflexive at $\mathbb{b}$ with $\mathbb{a}, \mathbb{b}, \mathbb{c} \!\in\! X$.
Then the concatenation $\mu=\mu_{X} : \Path{X;\mathbb{a},\mathbb{b}} \times \Path{X;\mathbb{b},\mathbb{c}} \to \Path{X;\mathbb{a},\mathbb{c}}$ is well-defined and {\differentiabletxt}.
\end{thm}
\begin{proof}
For any two plots $P : \real^{k} \supset U \to \Path{X;\mathbb{a},\mathbb{b}}$, $Q : \real^{\ell} \supset V \to \Path{X;\mathbb{b},\mathbb{c}}$ and any $\varphi \in \mathcal{F}(X)$, the adjoints $\widehat{P} : \real \times U \to X$ and $\widehat{Q} : \real \times V \to X$ of $P$ and $Q$ are smooth, and hence $\varphi{\comp}\widehat{P} : \real \times U \to \real$ and $\varphi{\comp}\widehat{Q} : \real \times V \to \real$ are {\smoothtxt} functions.
Let $\widehat{(P{\cdot}Q)} : \real \times U \times V \to X$ be the adjoint map of $P{\cdot}Q := \mu{\comp}(P \times Q) : U \times V \to \Path{X;\mathbb{a},\mathbb{b}} \times \Path{X;\mathbb{b},\mathbb{c}} \to \Path{X;\mathbb{a},\mathbb{c}}$. 
The function $\varphi{\comp}\widehat{(P{\cdot}Q)}$ can be described as
\begin{equation}
\varphi{\comp}\widehat{(P{\cdot}Q)}(t,\mathbb{x},\mathbb{y}) = \begin{cases}\,
\varphi{\comp}\widehat{P}(2t,\mathbb{x}), &t\le\fracinline1/2,
\\[.5ex]\,
\varphi{\comp}\widehat{Q}(2t{-}1,\mathbb{y}), &t\ge\fracinline1/2,
\end{cases}\quad (t,\mathbb{x},\mathbb{y}) \in \real \times U \times V.
\tag{B}\label{eq:concat-plot}
\end{equation}
Since the adjoint map $\widehat{P}$ of $P$ is constant on $(-\infty,0] \times U$ and on $[1,\infty) \times U$, so is $\varphi{\comp}\widehat{P}$, and we obtain the following equation.
\begin{align*}
\pder{(\varphi{\comp}\widehat{P})}by{t} = \pder{(\varphi{\comp}\widehat{P})}by{x_{i}} = 0 \ \ \text{on} \ \ ((-\infty,0] \cup [1,\infty)) \times U \ \text{for} \ 1 \!\le\! i \!\le\! k.
\end{align*}
Similarly for the adjoint map $\widehat{Q}$ of $Q$, we obtain the following equation.
\begin{align*}
\pder{(\varphi{\comp}\widehat{Q})}by{t} = \pder{(\varphi{\comp}\widehat{Q})}by{y_{j}} = 0 \ \ \text{on} \ \ ((-\infty,0] \cup [1,\infty)) \times V \ \text{for} \ 1 \!\le\! j \!\le\! \ell.
\end{align*}
Hence by iterating partial differentials, we obtain
\begin{align*}&
\frac{\partial^{n+|I|}(\varphi{\comp}\widehat{P})}{\pdiff{t}^{n}\pdiff{\mathbb{x}}^{I}} = 0 \ \text{on} \ ((-\infty,0] \cup [1,\infty)) \times U \ \text{for} \ (n,I) \in \numeric \times \numeric^{k}, \ n\!+\!|I|>0,
\\[1ex]&
\frac{\partial^{n+|J|}(\varphi{\comp}\widehat{Q})}{\pdiff{t}^{n}\pdiff{\mathbb{y}}^{J}} = 0 \ \text{on} \ ((-\infty,0] \cup [1,\infty)) \times V \ \text{for} \ (n,J) \in \numeric \times \numeric^{\ell}, \ \text{if $n+|J|>0$,} 
\end{align*}
where $|I|= i_{1}+\cdots+i_{k}$ and $\partial\mathbb{x}^{I}=\partial{x}_{1}^{i_{1}}\cdots\partial{x}_{k}^{i_{k}}$ for $I\!=\!(i_{1},\ldots,i_{k})$, and $|J|=j_{1}+\cdots+j_{\ell}$ and $\partial\mathbb{y}^{J}=\partial{y}_{1}^{j_{1}}\cdots\partial{y}_{\ell}^{j_{\ell}}$ for $J\!=\!(j_{1},\ldots,j_{\ell})$.
Let $\Phi_{1}=\varphi{\comp}\mu(\widehat{P},k(\mathbb{b}))=\mu_{\real}(\varphi{\comp}\widehat{P},k(\varphi(\mathbb{b})))$ and $\Phi_{2}=\mu_{\real}(k(\varphi(\mathbb{b})),\varphi{\comp}\widehat{Q})$.
Then we obtain the following equations.\vspace{.5ex}
\begin{align*}&
\frac{\partial^{n+|I|+|J|}\Phi_{1}}{\pdiff{t}^{n}\pdiff{\mathbb{x}}^{I}\pdiff{\mathbb{y}}^{J}} = 0 \ \text{on} \ ((-\infty,0] \cup [\fracinline1/2,\infty)) \times U \ \ \text{and}
\\[1ex]&
\frac{\partial^{n+|I|+|J|}\Phi_{2}}{\pdiff{t}^{n}\pdiff{\mathbb{x}}^{I}\pdiff{\mathbb{y}}^{J}} = 0 \ \text{on} \ ((-\infty,\fracinline1/2] \cup [1,\infty)) \times V
\end{align*}
for $(n,I,J) \in \numeric \times \numeric^{k} \times \numeric^{\ell}\!$, \,$n+|I|+|J|>0$.
Since $\varphi{\comp}\widehat{(P{\cdot}Q)}(t,\mathbb{x},\mathbb{y})$ $=$ $\mu_{\real}(\varphi{\comp}\widehat{P},\varphi{\comp}\widehat{Q})(t,\mathbb{x},\mathbb{y})$ $=$ $\Phi_{1}(t,\mathbb{x},\mathbb{y}) + \Phi_{2}(t,\mathbb{x},\mathbb{y}) - \varphi{\comp}k(\mathbb{b})$, $(t,\mathbb{x},\mathbb{y}) \!\in\! \real \times U \times V$ by the equation (\ref{eq:concat-plot}), the composition $\varphi{\comp}\widehat{(P{\cdot}Q)}$ is also a smooth function satisfying 
\begin{equation*}
\frac{\partial^{n+|I|+|J|}(\varphi{\comp}\widehat{(P{\cdot}Q)})}{\pdiff{t}^{n}\pdiff{\mathbb{x}}^{I}\pdiff{\mathbb{y}}^{J}} = 0 \ \text{on} \ ((-\infty,0] \cup \{\fracinline1/2\} \cup [1,\infty)) \times U \times V 
\end{equation*}
for $(n,I,J) \in \numeric \times \numeric^{k} \times \numeric^{\ell}$, $n+|I|+|J|>0$.
Therefore, by the reflexivity at a point $\mathbb{b} \!\in\! X$, we can deduce, using a similar argument given to show the well-definedness of $\mu(u,v)$, that the map $\widehat{(P{\cdot}Q)} : \real \times U \times V \to X$ is a plot, which means $\mu{\comp}(P \times Q) : U \times V \to \Path{X}$ is a plot, and hence $\mu$ is {\differentiabletxt}.
\end{proof}

There also is an obvious smooth homotopy $\mu{\comp}(\iota \times_{X} \id) \sim \id \sim \mu{\comp}(\id \times_{X} \iota)$.

\begin{cor}
If $X$ is reflexive, then $\mu : \Path{X} \times_{X} \Path{X} \to \Path{X}$ is smooth, and hence $(\Path{X},X)$ is an internal H-category in \diffeology{}.
If $X$ is reflexive at the base point $\mathbb{\ast}$, then the concatenation of $\Loop{X}=\Path{X;\mathbb{\ast},\mathbb{\ast}}$ is smooth.
\end{cor}

\begin{expl}
Let $X$ be a manifold, or a smooth CW complex of finite dimension whose base point is in the top cell.
Then the concatenation of $\Loop{X}$ is smooth.
\end{expl}

\begin{thm}\label{thm:reflexive-h-unital}
If $X$ is a reflexive diffeological space, then $(\Path{X},X)$ is an internal smooth h-unital $A_{\infty}$-category in \diffeology{}.
\end{thm}
\begin{proof}
Firstly, we introduce the space of concatenations as follows:
$$
E_{n} = \{\, (r_{1},r_{2},\ldots,r_{n}) \in \real^{n-1} \midvert 0 < r_{i} \ (1 \!\le\! i \!\le\! n), \ r_{1}+r_{2}+\cdots+r_{n} = 1 \,\}
$$
which is a convex open set in an affine space of dimension $n{-}1$:
$$
H^{n-1} : x_{1}+x_{2}+\cdots+x_{n}=1.
$$
Let $\Path{X}^{n} = \Path{X} \times_{X} \cdots \times_{X} \Path{X}$.
Then we define $\widehat\beta_{n} : E_{n} \times \Path{X}^{n} \hooklongrightarrow \Path{X}$ as
$$
\widehat\beta_{n}(r_{1},\dots,r_{n};u_{1},\ldots,u_{n-1})=v \in \Path{X},
$$
which is defined by
$$
v(t)=\begin{cases}
u_{1}(\frac{t}{r_{1}}),&t \!\le\! v_{1},
\\[1.5ex]
u_{i}(\frac{t{-}v_{i-1}}{r_{i}}),&v_{i-1} \!\le\! t \!\le\! v_{i}, \ 1 \!<\! i \!<\! n,
\\[1.5ex]
u_{n}(\frac{t{-}v_{n-1}}{r_{n}}),&v_{n-1} \!\le\! t,
\end{cases}\quad \text{$t \in \real$,}
$$
where $v_{i}=r_{1}+\cdots+r_{i}$, $1 \!\le\! i \!\le\! n$.
It is not very hard to show that $\widehat\beta_{n}$ is well-defined and also smooth by using a similar arguments to the proof of Theorem \ref{thm:concatenation}, and we leave it to the reader.
Thus by taking adjoint, we obtain a smooth map $\beta_{n} : E_{n} \to C^{\infty}(\Path{X}^{n},\Path{X})$.

Secondly, let $\partial^{E}_{k} : E_{r} \times E_{s} \to E_{n}$, $r{+}s=n{+}1$, be the smooth map defined by
$$\partial^{E}_{k}(x_{1},\dots,x_{r};y_{1},\dots,y_{s})=(x_{1},\dots,x_{k-1},x_{k}{\cdot}y_{1},\dots,x_{k}{\cdot}y_{s},x_{k+1},\dots,x_{r}).
$$
Let $\mathbb{e}_{n} = \frac1n(1,1,\ldots,1) \in E_{n}$.
Then we define $\phi_{n} : K_{n} \to E_{n}$ inductively by 
\par\vskip1ex\noindent
\begin{enumerate*}
\item $\phi_{n} : K_{n} \ni \mathbb{b}_{n} \mapsto \mathbb{e}_{n} \in E_{n}$, and
\hitem $\phi_{n}{\comp}\partial_{k} = \partial^{E}_{k}{\comp}(\phi_{r} \times \phi_{s})$.
\end{enumerate*}%
\par\vskip1ex\noindent
By (1), $\phi_{2} : K_{2} \!=\! \{(0,1)\} \to \{\mathbb{e}_{2}\} \subset E_{2}$ is the trivial map, since $\mathbb{b}_{2}\!=\!(0,1)$.
(2) determines a smooth map $\widehat\phi_{k,r,s} : L_{k}(r,s) \to E_{n}$, since $\partial_{k} : K_{r} \times K_{s} \to L_{k}(r,s)$ is a diffeomorphism.
There is a smooth extension $\phi_{k,r,s} : L_{k}(r,s) \join \{\mathbb{b}_{n}\} \to E_{n}$ of \,$\widehat\phi_{k,r,s}$, $(k,r,s) \in A(n)$, such that $\phi_{k,r,s}(\mathbb{b}_{n}) = \mathbb{e}_{n}$, since $E_{n}$ is star-shaped w.r.t.\ $\mathbb{e}_{n} \!\in\! E_{n}$.
As $|\mathbb{K}(n)|$ is the colimit of $|\mathbb{L}_{k}(r,s) \join \{\mathbb{b}_{n}\}|$ and the identity map $|\mathbb{L}_{k}(r,s) \join \{\mathbb{b}_{n}\}| \to L_{k}(r,s) \join \{\mathbb{b}_{n}\}$ is a smooth bijection, smooth maps $\phi_{k,r,s}$, $(k,r,s) \in A(n)$, give a smooth map $\phi_{n} : |\mathbb{K}(n)| \to E_{n}$. 

Finally, smooth maps $M_{n}=\beta_{n}{\comp}\phi_{n}$ satisfy (0), (1) and (2') in Definition \ref{defn:internal-higher-associativity}, which determines a smooth $A_{\infty}$-form on $(\Path{X},X)$, and hence $(\Path{X},X)$ is an internal smooth $A_{\infty}$-category in \diffeology{}.
\end{proof}

\begin{cor}\label{cor:reflexive-h-unital-loop}
If a diffeological space $X$ is reflexive at the base point,
then the diffeological loop space $\Loop{X}$ is a smooth h-unital $A_{\infty}$-space in \diffeology{}.
\end{cor}

\begin{cor}\label{cor:CW-h-unital-loop}
Let $X$ be a manifold, or a smooth CW complex of finite dimension with base point in the interior of the top cell.
Then $\Loop{X}$ is a smooth h-unital $A_{\infty}$-space in \diffeology{}.
\end{cor}

\appendix

\section{Path space with stable concatenation}

In this section, we modify a concatenation to obtain a smooth $A_{\infty}$-form on $\Path{X}$ without assuming reflexivity on $X$. 
The internal precategory $\Path{X}$ is equipped with the following stable concatenation for a fixed small $\varepsilon\!>\!0$.
\begin{align*}&
\mu_{\varepsilon} : \Path{X} \times_{X} \Path{X} \to \Path{X}
\\[-1ex]&\qquad\qquad
\iff \mu_{\varepsilon}(u,v)(t) = 
\begin{cases}
\,u(\frac{2t}{1-\varepsilon}),& t < \frac{1+\varepsilon}2,
\\[2ex]
\,v(\frac{2t-1-\varepsilon}{1-\varepsilon}),&t > \frac{1-\varepsilon}2,
\end{cases} \ \ t \in \real.
\end{align*}\vskip.5ex\noindent
Then $\mu_{\varepsilon}$ is stable when $\frac{1-\varepsilon}2 \!\le\! t \!\le\! \frac{1+\varepsilon}2$, and is smooth for any diffeological space $X$. 

\begin{thm}\label{thm:homotopy-unital-loop}
$(\Path{X},X)$ is an internal h-unital $A_{\infty}$-category in \diffeology{}.
\end{thm}
\begin{proof}
Firstly, we introduce the space of stable concatenations as follows:
$$D_{n}=\{\,(r_{1},\dots,r_{n};\varepsilon_{1},\dots,\varepsilon_{n-1}) \!\in\! (0,1)^{2n-1} \midvert r_{1}+\cdots{+}r_{n}+\varepsilon_{1}+\cdots+\varepsilon_{n-1}=1\,\}$$
which is a convex open set in the following affine space of dimension $2n{-}2$: 
$$
H^{2n-2} : x_{1}+\cdots+x_{n}+x_{n+1}+\cdots+x_{2n-1}=1.
$$
We define
$\widehat\alpha_{n} : D_{n} \times \Path{X}^{n} \to \Path{X}$ as
\begin{align*}&
\widehat\alpha_{n}(r_{1},\dots,r_{n};\varepsilon_{1},\dots,\varepsilon_{n-1},r_{n};u_{1},\dots,u_{n})=v \in \Path{X},
\end{align*}
which is defined by
\begin{align*}
v(r) = 
\begin{cases}\,
u_{1}(\frac{r}{r_{1}}),&r \!<\! s_{1}{+}\varepsilon_{1},
\\[1.5ex]\,
u_{i}(\frac{r{-}s_{i-1}{-}\varepsilon_{i-1}}{r_{i}}), &s_{i-1} \!<\! r \!<\! s_{i}{+}\varepsilon_{i}, \ 1 \!<\! i \!<\! n,
\\[1.5ex]\,
u_{n}(\frac{r{-}s_{n-1}{-}\varepsilon_{n-1}}{r_{n}}).&s_{n-1} \!<\! r,
\end{cases}\quad \text{$r \in \real$,}
\end{align*}
where $s_{i}=r_{1}+\cdots+r_{i}+\varepsilon_{1}+\cdots+\varepsilon_{i-1}$, $1 \!\le\! i \!\le\! n$, and hence $s_{i}\!-\!s_{i-1}\!-\!\varepsilon_{i-1}=r_{i}$, $1 \!\le\! i \!\le\! n$. 
Since the open sets $(-\infty,s_{1}{+}\varepsilon_{1}) \supset (-\infty,s_{1}]$, $(s_{i-1},s_{i}{+}\varepsilon_{i}) \supset (s_{i-1},s_{i}]$, $1 \!<\! i \!<\! n$ and $(s_{n-1},\infty)$ cover entire $\real$, $\widehat\alpha_{n}$ is well-defined by definition and is also smooth.
Thus by taking adjoint, we obtain a smooth map $\alpha_{n} : D_{n} \to C^{\infty}(\Path{X}^{n},\Path{X})$.

Secondly, let $\partial^{D}_{k} : D_{r} \times D_{s} \to D_{n}$, $r{+}s=n{+}1$, be the smooth map defined by
\begin{align*}
&\partial^{D}_{k}(x_{1},\dots,x_{r};\varepsilon_{1},\dots,\varepsilon_{r-1};y_{1},\dots,y_{s};\eta_{1},\dots,\eta_{r-1})
\\&\quad
=(x_{1},\dots,x_{k-1},x_{k}{\cdot}y_{1},\dots,x_{k}{\cdot}y_{s},x_{k+1},\dots,x_{r};\varepsilon_{1},\dots,\varepsilon_{k-1};x_{k}{\cdot}\eta_{1},\dots,x_{k}{\cdot}\eta_{s-1};\varepsilon_{k},\dots,\varepsilon_{r-1}).
\end{align*}
Let $\mathbb{d}_{n}=\frac1{3n-1}(2,\dots,2;1,\dots,1) \in D_{n}$.
Then we define $\psi_{n} : K_{n} \to D_{n}$ inductively by 
\par\vskip.5ex\noindent
\begin{enumerate*}
\item $\psi_{n} : K_{n} \ni \mathbb{b}_{n} \mapsto \mathbb{d}_{n} \in D_{n}$\quad and  
\hitem $\psi_{n}{\comp}\partial_{k} = \partial^{D}_{k}{\comp}(\psi_{r} \times \psi_{s})$.
\end{enumerate*}%
\par\vskip.5ex\noindent
By (1), $\psi_{2} : K_{2} \!=\! \{(0,1)\} \to \{\mathbb{d}_{2}\} \subset D_{2}$ is the trivial map, since $\mathbb{b}_{2}\!=\!(0,1)$.
(2) determines a smooth map $\widehat\psi_{k,r,s} : L_{k}(r,s) \to D_{n}$, since $\partial_{k} : K_{r} \times K_{s} \to L_{k}(r,s)$ is a diffeomorphism.
There is a smooth extension $\psi_{k,r,s} : L_{k}(r,s) \join \{\mathbb{b}_{n}\} \to D_{n}$ of \,$\widehat\psi_{k,r,s}$, $(k,r,s) \in A(n)$, such that $\psi_{k,r,s}(\mathbb{b}_{n}) = \mathbb{d}_{n}$, since $D_{n}$ is star-shaped w.r.t.\ $\mathbb{d}_{n} \!\in\! D_{n}$. 
Since $|\mathbb{K}(n)|$ is the colimit of $|\mathbb{L}_{k}(r,s) \join \{\mathbb{b}_{n}\}|$ and the identity map $|\mathbb{L}_{k}(r,s) \join \{\mathbb{b}_{n}\}|$ $\to$ $L_{k}(r,s) \join \{\mathbb{b}_{n}\}$ is a smooth bijection, smooth maps $\psi_{k,r,s}$, $(k,r,s) \in A(n)$, give rise to a smooth map $\psi_{n} : |\mathbb{K}(n)| \to D_{n}$.

Finally, smooth maps $M_{n}=\alpha_{n}{\comp}\psi_{n}$ satisfy (0), (1) and (2') in Definition \ref{defn:internal-higher-associativity}, which determines a smooth $A_{\infty}$-form on $(\Path{X},X)$, and hence $(\Path{X},X)$ is an internal smooth $A_{\infty}$-category in \diffeology{}.
\end{proof}

\begin{cor}\label{cor:homotopy-unital-loop}
$\Loop{X}$ is a smooth h-unital $A_{\infty}$-space in \diffeology{}.
\end{cor}

\section*{Acknowledgements}
This research was supported by Grant-in-Aid for Scientific Research (S) \#17H06128 and Exploratory Research \#18K18713 from Japan Society for the Promotion of Science.

%
%

\bibliographystyle{alpha}
\bibliography{2020diff}
\vskip1ex

\end{document}